%% file: embeddings-arxiv.tex
\documentclass[reqno]{amsart}

\usepackage[english]{babel}
\usepackage[utf8x]{inputenc}
\usepackage{latexsym,amssymb,amsmath,amsopn,amsthm,amscd,amsfonts,url}

\usepackage[pdftex,pdfcreator={LaTeX},pdfproducer={PDFLaTeX},colorlinks,
	bookmarks,
	pdftitle={},
	pdfauthor={},
	pdfkeywords={},
	linkcolor=blue,                  
	citecolor=red,                   
	filecolor=red,                   
	urlcolor=cyan                    
]{hyperref}

\newcommand{\dateline}[3]{}
\newcommand{\MSC}[2][2020]{\subjclass[#1]{#2}}
\newcommand{\Copyright}[1]{}
\newcommand{\doi}[1]{\href{https://doi.org/#1}{\texttt{doi:#1}}}
\newcommand{\arxiv}[1]{\href{https://arxiv.org/abs/#1}{\texttt{arXiv:#1}}}

\newtheorem{definition}{Definition}
\newtheorem{theorem}[definition]{Theorem}

\newtheorem{lemma}[definition]{Lemma}

\newtheorem{remark}[definition]{Remark}

\newcommand{\Title}[2][]{\title[#1]{#2}}
\newcommand{\ejcbrk}{}

\input{embeddings-header.tex}

\begin{document}
\input{embeddings-title.tex}
\author[\authorabbrv]{\authorname}
\address{\authorname,\newline
	\authortextone.
	\newline also affiliated with: \newline
	\authortexttwo.
}
\email{\authoremail}

\input{embeddings-abstract.tex}
\maketitle

\input{embeddings.tex}


\input{embeddings-bib.tex}
\end{document}

%% file: embeddings-header.tex
\usepackage{graphicx}

\usepackage{algorithm}
\usepackage{algorithmicx,algpseudocode}
\usepackage{tikz}
\usepackage{array}
\usepackage[small]{caption}
\usetikzlibrary{calc}
\usetikzlibrary{shapes}
\usetikzlibrary{shapes.misc}

\newcommand{\A}{\ensuremath{\mathcal{A}}}
\newcommand{\C}{\ensuremath{\mathbb{C}}}
\newcommand{\D}{\ensuremath{\mathcal{D}}}
\newcommand{\F}{\ensuremath{\mathbb{F}}}
\newcommand{\G}{\ensuremath{\Gamma}}
\newcommand{\I}{\ensuremath{\mathcal{I}}}
\newcommand{\J}{\ensuremath{\mathcal{J}}}
\newcommand{\Q}{\ensuremath{\mathbb{Q}}}
\newcommand{\R}{\ensuremath{\mathcal{R}}}
\newcommand{\RR}{\ensuremath{\mathbb{R}}}
\newcommand{\Ss}{\ensuremath{\mathcal{S}}}
\newcommand{\M}{\ensuremath{\mathcal{M}}}
\newcommand{\Z}{\ensuremath{\mathbb{Z}}}
\newcommand{\Id}{\ensuremath{\operatorname{Id}}}

\newcommand{\Cay}{\ensuremath{\operatorname{Cay}}}
\newcommand{\Cyc}{\ensuremath{\operatorname{Cyc}}}
\newcommand{\Had}{\ensuremath{\operatorname{Had}}}
\newcommand{\Pair}{\ensuremath{\operatorname{Pair}}}
\newcommand{\GQ}{\ensuremath{\operatorname{GQ}}}
\newcommand{\GH}{\ensuremath{\operatorname{GH}}}
\newcommand{\PG}{\ensuremath{\operatorname{PG}}}
\newcommand{\SRG}{\ensuremath{\operatorname{SRG}}}
\newcommand{\Conf}{\ensuremath{\operatorname{Conf}}}
\newcommand*{\ind}{\text{\usefont{U}{bbold}{m}{n}1}}
\newcommand{\noop}[1]{}

\algrenewcommand{\algorithmicrequire}{\textbf{Input:}}

%% file: embeddings-title.tex
\dateline{Jan 1, 2024}{Jan 2, 2024}{TBD}

\MSC{05E30}


\Copyright{The author. Released under the CC BY license (International 4.0).}

%

\Title[Eigenspace embeddings of imprimitive association schemes]{Eigenspace embeddings\\ of imprimitive association schemes}



\newcommand{\authorname}{Jano\v{s} Vidali}
\newcommand{\authorabbrv}{J.\ Vidali}
\newcommand{\authortextone}{Faculty of Mathematics and Physics, University of Ljubljana, Slovenia}
\newcommand{\authoremail}{janos.vidali@fmf.uni-lj.si}
\newcommand{\authortexttwo}{Institute of Mathematics, Physics and Mechanics, Ljubljana, Slovenia}





%% file: embeddings-abstract.tex

\begin{abstract}
For a given symmetric association scheme $\A$ and its eigenspace $S_j$
there exists a mapping of vertices of $\A$ to unit vectors of $S_j$,
known as the spherical representation of $\A$ in $S_j$,
such that the inner products of these vectors
only depend on the relation between the corresponding vertices;
furthermore, these inner products only depend on the parameters of $\A$.
We consider parameters of imprimitive association schemes
listed as open cases in the list of parameters for quotient-polynomial graphs
recently published by Herman and Maleki,
and study embeddings of their substructures into some eigenspaces
consistent with spherical representations of the putative association schemes.
Using this,
we obtain nonexistence for two parameter sets for $4$-class association schemes
and one parameter sets for a $5$-class association scheme
passing all previously known feasibility conditions,
as well as uniqueness for two parameter sets for $5$-class association schemes.
\end{abstract}

%% file: embeddings.tex
\section{Introduction}

Association schemes were first introduced
within the theory of experimental design,
however, since Delsarte~\cite{d73a},
they have been primarily studied as combinatorial objects of their own,
representing the basic underlying structures in various fields
such as coding theory, design theory, and finite geometry.
Much of the research on association schemes
has been focused on some special cases,
such as strongly regular graphs (i.e., $2$-class association schemes),
distance-regular graphs (corresponding to $P$-polynomial association schemes)
and $Q$-polynomial association schemes.
Nevertheless, even for these subfamilies,
a complete classification is still a widely open problem.
Tables of feasible parameters for various families of association schemes
have been compiled,
in particular,
by Brouwer et al.~\cite{b11,b13,bcn89}
for strongly regular and distance-regular graphs,
by Van Dam~\cite{vd99} for three-class association schemes,
and by Williford~\cite{gvw21,w17} for $Q$-polynomial association schemes.
Recently,
two new surveys of feasible parameter sets of association schemes
have been compiled by Herman and Maleki:
one for association schemes with noncyclotomic eigenvalues~\cite{hm23}
and one for quotient-polynomial graphs~\cite{hm23a,hm24}.

Contributions to the classification of association schemes
come in the form of new constructions
and characterizations of association schemes
with a certain parameter set or belonging to a family of parameter sets
-- in particular,
it may be possible to prove that there is a unique association scheme
with a given parameter set (uniqueness proof),
or that there are none (nonexistence proof).
Many families and sporadic examples of association schemes are known,
and constructing new ones,
particularly in the more studied subfamilies,
has proved to be increasingly difficult.
On the other hand,
there are many parameter sets which pass the known feasibility conditions,
but no corresponding association scheme has been constructed;
there are also many cases
when one or more corresponding association schemes are known,
but it is not known whether there are any more.

One of the techniques that can be used to study association schemes
is to study their spherical representations in their eigenspaces.
Bannai, Bannai and Bannai~\cite{bbb08} have used this technique
to prove uniqueness of two association schemes arising from spherical codes;
more recently,
Gavrilyuk, Suda et al.~have used a similar technique
to prove uniqueness of an association scheme
related to the Witt design on $11$ points~\cite{gs23},
and of a non-symmetric commutative association scheme
arising from a set of equiangular lines in $\C^8$~\cite{glms25}.
In the present paper,
we apply such a technique to study imprimitive association schemes
with parameters which are listed as open cases
in the aforementioned list of parameter sets of association schemes
corresponding to quotient-polynomial graphs.
We first apply the known feasibility conditions
and find numerous cases when they either rule out a parameter set,
or there is a known example (see Appendix~\ref{app:known}).
Then,
using software~\cite{v19,v25}
developed on top of the SageMath computer algebra system~\cite{sage24},
we conduct some computer searches
and conclude nonexistence for three of the cases
that satisfy the known feasibility conditions,
and uniqueness for two more cases.

\section{Preliminaries}

In this section we review some basic definitions and concepts.
See Brouwer, Cohen and Neumaier~\cite{bcn89} for further details.

Let $X$ be a set of $n$ {\em vertices},
and $R \subseteq X^2$ a binary relation on $X$.
The matrix $A \in \{0, 1\}^{X \times X}$
such that $A_{xy} = 1$ if and only if $(x, y) \in R$
is called the {\em adjacency matrix} of the relation $R$.
If $R$ is an irreflexive relation,
then the pair $\G = (X, R)$ is called a (simple, directed) {\em graph}
-- such a graph has the set $X$ as its vertex set
and the set $R$ as its arc set (i.e., the set of its directed edges),
and $A$ is its adjacency matrix.
In the case when $R$ is a symmetric relation,
we will understand the graph $\G$ to be undirected,
and its edges are precisely the unordered pairs $\{x, y\}$
such that $(x, y) \in R$.
For a subset $Y \subseteq X$,
we define the {\em induced subgraph} of $\G$ on $Y$ as $\G|_Y = (Y, R|_Y)$,
where $R|_Y = \{(x, y) \in R \mid x, y \in Y\}$
is the restriction of the relation $R$ onto the subset $Y$.

Let $\R = \{R_i \mid i \in \I\}$,
where $\I$ is an index set of size $d+1$ for some $d$,
be a partition of $X^2$ such that
$\Id_X := \{(x, x) \mid x \in X\} \in \R$ and $\emptyset \not\in \R$
-- i.e., $\R$ is a set of binary relations on $X$
containing the identity relation
such that each pair of vertices of $X$ lies in precisely one relation of $\R$.
A {\em relation scheme} is defined by a pair $\A = (X, \R)$.
A non-identity relation of $\R$ is also called a {\em class},
so we may refer to $\A$ as a {\em $d$-class relation scheme}.
Customarily,
we will have $\I = \{0, 1, \dots, d\}$ and $R_0 = \Id_X$,
although we may occasionaly deviate from this convention.
A relation scheme may be concisely represented by its {\em relation matrix}
$M \in \I^{X \times X}$ satisfying $(x, y) \in R_{M_{xy}}$ ($x, y \in X$).

An {\em isomorphism} between relation schemes
$\A = (X, \R)$ and $\A' = (X', \R')$
is a pair $(\phi, \psi)$ of bijective maps
$\phi: X \to X'$ and $\psi: \R \to \R'$
such that for each pair of vertices $x, y \in X$
and for each relation $R \in \R$
we have $(x, y) \in R$ if and only if $(x^\phi, y^\phi) \in R^\psi$.
The relation schemes $\A$ and $\A'$ are {\em isomorphic}
if there exists such an isomorphism.
An {\em automorphism} of $\A$ is an isomorphism between $\A$ and itself.

If all the relations of $\R$ are symmetric,
then $\A$ is called a {\em symmetric relation scheme}.
For a subset $Y \subseteq X$,
we define the {\em induced subscheme} of $\A$ on $Y$ as $\A|_Y = (Y, \R|_Y)$,
where
$\R|_Y = \{R|_Y \mid R \in \R\} \setminus \{\emptyset\}$
is the restriction of the partition $\R$ onto the subset $Y$.
Note that $(\Id_X)|_Y = \Id_Y \in \R|_Y$, so $\A|_Y$ is also a relation scheme.
Clearly, if $\A$ is symmetric, $\A|_Y$ is symmetric as well.

Suppose that $\A = (X, \R)$ is a symmetric relation scheme
with the additional property
that there exist numbers $p^h_{ij}$ ($h, i, j \in \I$)
such that for each pair $(x, y) \in R_h$,
there are precisely $p^h_{ij}$ vertices $z \in X$
with $(x, z) \in R_i$ and $(z, y) \in R_j$.
Then $\A$ is called a (symmetric) {\em association scheme},
and the numbers $p^h_{ij}$ ($h, i, j \in \I$)
are its {\em intersection numbers}.
The number $k_i := p^0_{ii}$ is the {\em valency} of the relation $R_i$
($i \in \I$)
-- i.e., for each vertex $x \in X$,
there exist precisely $k_i$ vertices $y \in X$ such that $(x, y) \in R_i$.
From now on, we will assume that $\A$ is an association scheme.

Let $A_i$ be the adjacency matrix of the relation $R_i$ ($i \in \I$)
-- then we say that $A_i$ ($i \in \I$) are the adjacency matrices
of the association scheme $\A$.
We also define the corresponding graphs $\G_i = (X, R_i)$
($i \in \I$, $R_i \ne \Id_X$).
Note that $A_i A_j = \sum_{h \in \I} p^h_{ij} A_h$ ($i, j \in \I$) holds.
In particular,
since the adjacency matrices of a symmetric association scheme are symmetric,
they can be simultaneously diagonalized,
giving a decomposition of $\RR^X$
as a direct sum of $d+1$ common eigenspaces
forming a set $\Ss = \{S_j \mid j \in \J\}$,
where $\J$ is an index set of size $d+1$.
Note that the all-ones matrix $J = \sum_{i \in \I} A_i$
has an eigenvalue $n$ with multiplicity $1$
and its corresponding eigenspace is $\langle \ind_X \rangle$,
i.e., the one-dimensional subspace of $\RR^X$ spanned by the all-ones vector;
this subspace is also an eigenspace of $A_i$ ($i \in \I$)
for the eigenvalue $k_i$.
Therefore, $\langle \ind_X \rangle \in \Ss$.
Customarily,
we will have $\J = \{0, 1, \dots, d\}$ and $S_0 = \langle \ind_X \rangle$,
although we may, again, occasionaly deviate from this convention
(in particular,
it may happen that $\I$ and $\J$ are different sets of the same cardinality).

Let $E_j \in \RR^{X \times X}$ ($j \in \J$)
be the projector matrix onto the eigenspace $S_j$
-- these matrices are called the {\em minimal idempotents} of $\A$.
We note that the {\em Bose-Mesner algebra} of $\A$, i.e.,
the algebra generated by the basis of adjacency matrices $\{A_i \mid i \in \I\}$
with respect to ordinary matrix addition and multiplication,
has a second basis $\{E_j \mid j \in \J\}$~\cite[\S 2.2]{bcn89}.
Therefore,
there exist matrices $P \in \RR^{\J \times \I}$ and $Q \in \RR^{\I \times \J}$
(called the {\em eigenmatrix} and the {\em dual eigenmatrix}, respectively)
such that $A_i = \sum_{j \in \J} P_{ji} E_j$ ($i \in \I$)
and $E_j = {1 \over n} \sum_{i \in \I} Q_{ij} A_i$ ($j \in \J$).
We note that for each choice of $i \in \I$,
the value $P_{ji}$ ($j \in \J$)
corresponds to the eigenvalue of $A_i$ on the eigenspace $S_j$
(thus covering all distinct eigenvalues of $A_i$,
possibly with some repetitions).
In particular, we have $P_{0i} = k_i$ ($i \in \I$).
In the case when there exists a bijection $\iota: \I \to \J$,
$\iota: i \to i'$ ($i \in \I$)
such that $P_{i'j} = Q_{ij'}$ holds for all $i, j \in \I$,
we say that the association scheme $\A$ is {\em formally self-dual}.

Since the Bose-Mesner algebra $\M$ is also closed
under the entrywise multiplication of matrices
(denoted by $\circ$, also known as {\em Schur} or {\em Hadamard multiplication}),
it follows that there exist numbers $q^h_{ij}$ $(h, i, j \in \J)$,
known as the {\em Krein parameters},
such that $E_i \circ E_j = {1 \over n} \sum_{h \in \J} q^h_{ij} E_h$.
These numbers are nonnegative (cf.~\cite[Theorem~2.3.2]{bcn89}),
but not necessarily integral or rational,
yet they exhibit properties similar to those
of the intersection numbers of an association scheme
-- there is a formal duality between the two.
We also define the number $m_j := q^0_{jj}$
as the {\em multiplicity} of the eigenspace $S_j$ ($j \in \J$)
-- i.e.,
it corresponds to the dimension of $S_j$
and is therefore a positive integer.
Note that $Q_{0j} = m_j$ ($j \in \J$) also holds.

Any of the parameter sets
$\{p^h_{ij} \mid h,i,j \in \I\}$, $P$, $Q$ and $\{q^h_{ij} \mid h,i,j \in \J\}$
uniquely determines the others,
but not necessarily an association scheme itself
-- for any given parameter set,
there may be one or more association schemes, or none at all.

An {\em imprimitivity set} of the association scheme $\A = (X, \R)$
is a set of relation indices $\tilde{0} \subseteq \I$
such that $R_{\tilde{0}} := \bigcup_{i \in \tilde{0}} R_i$
is an equivalence relation partitioning the vertex set $X$
into the set of equivalence classes
$\tilde{X} := X / R_{\tilde{0}} = \{X_\ell \mid \ell = 1, 2, \dots, \tilde{n}\}$.
We note that
$|X_\ell| = \sum_{i \in \tilde{0}} k_i =: \overline{n}$ ($1 \le \ell \le \tilde{n}$)
and $n = \overline{n} \cdot \tilde{n}$.
Furthermore,
for each equivalence class $X_\ell$ ($1 \le \ell \le \tilde{n}$),
the induced subscheme $\A|_{X_\ell}$ is an association scheme
with intersection numbers
$\overline{p}^h_{ij} = p^h_{ij}$ ($h, i, j \in \tilde{0}$).
The association scheme $\A$ is called {\em imprimitive}
(cf.~\cite[\S 2.4]{bcn89})
if there exists a nontrivial imprimitivity set $\tilde{0}$
(i.e., $\{0\} \subset \tilde{0} \subset \I$, where $R_0 = \Id_X$).

An imprimitivity set $\tilde{0}$ also determines
an equivalence relation $\sim$ on $\I$ defined by
$$
h \sim j \iff \exists i \in \tilde{0}.\ p^h_{ij} \ne 0 \quad (h, j \in \I).
$$
Note that $\tilde{0}$ is an equivalence class of $\sim$,
and we define $\tilde{\imath}$ as the equivalence class of $\sim$
containing $i \in \I$.
This allows us to define the {\em quotient scheme}
$\tilde{\A} = \A / \tilde{0} = (\tilde{X}, \tilde{\R} = \{\tilde{R}_{\tilde{\imath}} \mid \tilde{\imath} \in \tilde{\I}\})$,
where $\tilde{\I} = \I/\!\sim$ and
$$
\tilde{R}_{\tilde{\imath}} = \{(\tilde{x}, \tilde{y}) \in \tilde{X}^2 \mid \exists x \in \tilde{x}, y \in \tilde{y}, i \in \tilde{\imath}. \ (x, y) \in R_i\} \quad (\tilde{\imath} \in \tilde{\I}).
$$

The quotient scheme $\tilde{\A}$ is an association scheme
with intersection numbers
$\tilde{p}^{\tilde{h}}_{\tilde{\imath}\tilde{\jmath}} = {1 \over \overline{n}} \sum_{i \in \tilde{\imath}} \sum_{j \in \tilde{\jmath}} p^h_{ij}$
for all $h \in \tilde{h}$
($\tilde{h}, \tilde{\imath}, \tilde{\jmath} \in \tilde{\I}$).
Thus,
the imprimitivity sets and the parameters
of the resulting subschemes and quotient scheme
only depend on the parameters of the parent association scheme.

Dually,
we may also define a {\em dual imprimitivity set}
as a set of eigenspace indices $\overline{0} \subseteq \J$
such that
$E_{\overline{0}} := {n \over \tilde{n}} \sum_{j \in \overline{0}} E_j$
is the adjacency matrix of an equivalence relation on $X$,
where $\tilde{n}$ equals the number of the resulting equivalence classes.
A dual imprimitivity set is nontrivial if
$\{0\} \subset \overline{0} \subset \J$,
where $S_0 = \langle \ind_X \rangle$.
It turns out that there is a one-to-one relationship
between (nontrivial) imprimitivity sets
and (nontrivial) dual imprimitivity sets,
i.e.,
$E_{\overline{0}}$ is the adjacency matrix of $R_{\tilde{0}}$,
where $\tilde{0}$ is the corresponding imprimitivity set.
In fact,
both imprimitivity sets and dual imprimitivity sets
can be recognized from the parameters of the association scheme.

Similarly as before,
the dual imprimitivity set $\overline{0}$
determines an equivalence relation $\simeq$ on $\J$
defined by
$$
h \simeq i \iff \exists j \in \overline{0}.\ q^h_{ij} \ne 0 \quad (h, i \in \J).
$$
Again we note that $\overline{0}$ is an equivalence class of $\simeq$,
and we define $\overline{\jmath}$ as the equivalence class of $\simeq$
containing $j \in \J$.
We find that the induced subschemes $\A|_{X_\ell}$ ($1 \le \ell \le \tilde{n}$)
have the set of eigenspaces
$\Ss|_{X_\ell} = \{S_{\overline{\jmath}}|_{X_\ell} \mid \overline{\jmath} \in \overline{\J}\}$,
where $\overline{\J} = \J/\!\simeq$ and
$$
S_{\overline{\jmath}}|_{X_\ell} = \left\{v|_{X_\ell} \ \middle| \ v \in \sum_{j \in \overline{\jmath}} S_j\right\} \quad (\overline{\jmath} \in \overline{\J}),
$$
i.e.,
the restriction to $X_\ell$
of the sum of eigenspaces with indices from $\overline{\jmath}$.
The Krein parameters of $\A|_{X_\ell}$ are then
$\overline{q}^{\overline{h}}_{\overline{\imath}\overline{\jmath}} = {1 \over \tilde{n}} \sum_{i \in \overline{\imath}} \sum_{j \in \overline{\jmath}} q^h_{ij}$
for all $h \in \overline{h}$
($\overline{h}, \overline{\imath}, \overline{\jmath} \in \overline{\J}$),
and its eigenmatrix $\overline{P}$ and dual eigenmatrix $\overline{Q}$ satisfy
$\overline{P}_{\overline{\jmath} i} = P_{ji}$ for all $j \in \overline{\jmath}$,
and
$\overline{Q}_{i \overline{\jmath}} = {1 \over \tilde{n}} \sum_{j \in \overline{\jmath}} Q_{ij}$
($i \in \tilde{0}, \overline{\jmath} \in \overline{\J}$).
In particular,
${Q_{ij} \over m_j} = {\overline{Q}_{i \overline{\jmath}} \over \overline{m}_{\overline{\jmath}}}$
holds for all $j \in \overline{\jmath}$,
where
$\overline{m}_{\overline{\jmath}} = \overline{q}^{\overline{0}}_{\overline{\jmath} \overline{\jmath}} = \overline{Q}_{0 \overline{\jmath}}$.
For the quotient scheme $\tilde{\A}$,
we find the set of eigenspaces
$\tilde{\Ss} = \{\tilde{S}_j = \{\tilde{v} \mid v \in S_j\} \mid j \in \overline{0}\}$,
where
$\tilde{v} = \left(\sum_{x \in \tilde{x}} v_x\right)_{\tilde{x} \in \tilde{X}} \in \RR^{\tilde{X}}$.
The Krein parameters of $\tilde{\A}$ are then
$\tilde{q}^h_{ij} = q^h_{ij}$ ($h, i, j \in \overline{0}$),
and its eigenmatrix $\tilde{P}$ and dual eigenmatrix $\tilde{Q}$ satisfy
$\tilde{P}_{j \tilde{\imath}} = {1 \over \overline{n}} \sum_{i \in \tilde{\imath}} P_{ji}$,
and $\tilde{Q}_{\tilde{\imath} j} = Q_{ij}$ for all $i \in \tilde{\imath}$
($\tilde{\imath} \in \tilde{\I}, j \in \overline{0}$).
In particular,
${P_{ji} \over k_i} = {\tilde{P}_{j \tilde{\imath}} \over \tilde{k}_{\tilde{\imath}}}$
holds for all $i \in \tilde{\imath}$,
where
$\tilde{k}_{\tilde{\imath}} = \tilde{p}^{\tilde{0}}_{\tilde{\imath} \tilde{\imath}} = \tilde{P}_{0 \tilde{\imath}}$.
Since the eigenmatrices are square matrices,
we see that $|\tilde{0}| = |\overline{\J}|$ and $|\overline{0}| = |\tilde{\I}|$.

\section{Quotient-polynomial graphs}
\label{sec:qpg}

Quotient-polynomial graphs (QPGs) were introduced by Fiol~\cite{f16},
whose results allow us to state the following definition.

\begin{definition}
Let $\G = (X, R)$ be an undirected graph with adjacency matrix $A$.
The graph $\G$ is {\em quotient-polynomial}
if the algebra generated by the powers of $A$ is the Bose-Mesner algebra
of an association scheme $\A = (X, \R = \{R_i \mid i \in \I\})$.
\end{definition}

Let $\G$ be a quotient-polynomial graph
and $\A$ the corresponding association scheme by the above definition.
Clearly, the graph $\G$ must be connected and regular,
as its adjacency matrix must be a sum of some adjacency matrices of $\A$.
Furthermore,
there exist polynomials $p_i$ ($i \in \I$) such that $A_i = p_i(A)$.
This shows that the notion of a quotient-polynomial graph
generalizes the notion of a distance-regular graph (see~\cite[\S4]{bcn89}),
as we drop the requirement on the degrees of these polynomials
and thus lose the equivalence
between the relations of $\A$ and distances in $\G$.

Given the parameters of an association scheme $\A$,
we can verify whether a sum of some adjacency matrices of $\A$
generates its Bose-Mesner algebra
-- each relation corresponding to such an adjacency matrix
thus gives us a quotient-polynomial graph.
A given association scheme may therefore correspond
to one or more quotient-polynomial graphs
(which need not be mutually non-isomorphic),
or none at all
(similarly to how an association scheme
may have multiple $P$-polynomial orderings,
thus corresponding to multiple distance-regular graphs).
In particular,
if a relation $\bigcup_{i \in \I'} R_i$ ($\I' \subset \I$)
of an association scheme $\A = (X, \{R_i \mid i \in \I\})$
gives rise to a quotient-polynomial graph,
this will also be true for the relation $\sum_{i \in \I'} R'_i$
of an association scheme $\A' = (X', \{R'_i \mid i \in \I\})$
with the same parameters as $\A$.

Herman and Maleki~\cite{hm24} define
a {\em relational} quotient-polynomial graph
as a quotient-polynomial graph $\G = (X, R)$
such that $R$ is a relation
of the corresponding $d$-class association scheme $\A$
with relation index set $\I = \{0, 1, \dots, d\}$.
In this case, we will assume $R_0 = \Id_X$ and $R_1 = R$.
For such an association scheme,
they define the {\em parameter array}
$$
[[k_1, k_2, \dots, k_d], [p^2_{11}, p^3_{11}, \dots p^d_{11}; p^3_{12}, p^4_{12}, \dots p^d_{12}; \dots; p^{d-1}_{1,d-2}, p^d_{1,d-2}; p^d_{1,d-1}]]
$$
and show that the remaining parameters of $\A$ can be computed from it.
This notation has been used to build a database of parameter arrays
(subject to limitations on number of classes, order and valency)
passing some basic feasibility conditions.
At the time of writing,
a subset of this database is available online~\cite{hm23a},
with tables for parameters arrays for QPGs
with $3$ classes of order at most $60$,
$4$ or $5$ classes of order at most $60$ and valency at most $12$,
and with $6$ classes of order at most $70$ and valency at most $12$
marking each entry either as infeasible (i.e., some further basic checks fail),
existing
(an association scheme with the corresponding parameters has been found)
or feasible
(all checks pass, but no example has been found).

We use the {\tt sage-drg} package~\cite{v18a,v19}
to perform more feasibility checks
for the parameter arrays marked as feasible.
For those parameter arrays which pass all the checks,
we attempt to identify known constructions.
The results are presented in Appendix~\ref{app:known}.
We also verify that the parameter sets
for association schemes with noncyclotomic eigenvalues in~\cite[\S4.3.2]{hm23}
pass the forbidden quadruple check (see~\cite[Corollary~4.2]{gvw21}),
as the other feasibility condition had already been verified.
For the parameter arrays which have neither been ruled out as infeasible
nor are there any known constructions for them,
we may use the technique presented in the following section
to study their feasibility.

\section{Eigenspace embeddings of association schemes}
\label{sec:embeddings}

Let $\A = (X, \R = \{R_i \mid i \in \I\})$ be an association scheme
with $n$ vertices, dual eigenmatrix $Q$ and multiplicities $m_j$ ($j \in \J$),
and $S_j$ ($j \in \J$) be one of its eigenspaces.
The entries of the corresponding minimal idempotent $E_j$
satisfy $(E_j)_{xy} = {Q_{ij} \over n}$ ($x, y \in X$)
if $(x, y) \in R_i$ ($i \in \I$).
Let $\ind_x \in \RR^X$ be the indicator vector of the vertex $x \in X$,
i.e., a unit vector with entry $1$ at index $x$ and entries $0$ elsewhere.
Then the vector $E_j \ind_x \in S_j$
is the orthogonal projection of $\ind_x$ onto the eigenspace $S_j$
and coincides with the column of $E_j$ at index $x$.
Consider the inner product of two such vectors:
for two vertices $x, y \in X$ such that $(x, y) \in R_i$,
we have
$$
\langle E_j \ind_x, E_j \ind_y \rangle = \ind_x^\top E_j^\top E_j \ind_y = \ind_x^\top E_j \ind_y = (E_j)_{xy} = {Q_{ij} \over n} .
$$
The inner product of $E_j \ind_x$ and $E_j \ind_y$
therefore only depends on the relation in which the vertices $x$ and $y$ are.
In particular, we have $\| E_j \ind_x \| = \sqrt{m_j \over n}$ ($x \in X$),
i.e.,
all the orthogonal projections of the vectors $\ind_x$ ($x \in X$) onto $S_j$
have the same norm.
Consequently,
the angle between two such projections only depends
on the relation in which the corresponding vertices are.
We may therefore define unit vectors
$u_x := \sqrt{n \over m_j} E_j \ind_x$ ($x \in X$),
and note that, for two vertices $x, y \in X$,
if $(x, y) \in R_i$ ($i \in \I$) holds,
then we have $\langle u_x, u_y \rangle = {Q_{ij} \over m_j}$.
The map $x \mapsto u_x$ is said to be a {\em spherical representation}
of the association scheme $\A$ in the eigenspace $S_j$.
A spherical representation is called {\em faithful} if it is injective.
Note that a spherical representation of $\A$ in the eigenspace $S_j$ is faithful
if and only if $S_j \ne \langle \ind_X \rangle$
and $j$ is not contained in any nontrivial dual imprimitivity set.

Let $\A' = (Y, \R')$ be a relation scheme with vertex set $Y \subseteq X$
and relations $\R' = \{R'_i \mid i \in \I'\}$ for some $\I' \subseteq \I$.
We say that the relation scheme $\A'$ {\em admits an embedding into $S_j$}
if there exist unit vectors $u'_x \in S_j$ ($x \in Y$)
such that for every two vertices $x, y \in Y$,
we have $\langle u'_x, u'_y \rangle = {Q_{ij} \over m_j}$
whenever $(x, y) \in R'_i$ ($i \in \I'$).
Clearly, if $R'_i = R_i|_Y$ holds for every $i \in \I'$,
then we have $\A' = \A|_Y$,
and we can just take $u'_x = u_x$ ($x \in Y$),
so the relation scheme $\A'$ admits an embedding into $S_j$.
Conversely, if no embedding of $\A'$ into $S_j$ exists,
then $\A'$ is not an induced subscheme of $\A$.

Given a relation scheme $\A'$,
we may therefore attempt to determine
the coefficients of the vectors $u'_x$ ($x \in Y$)
in terms of the coordinates with respect to an orthonormal basis
$\{e_h \mid h = 1, 2, \dots, m_j\}$ of $S_j$.
Let us write $u'_x = \sum_{h=1}^{m_j} a_{xh} e_h$,
where $a_{xh} \in \RR$ ($x \in Y$, $1 \le h \le m_j$).
We impose a linear order on the set $Y$,
say, by assuming $Y = \{1, 2, \dots, n'\}$,
and define a matrix $U := \{a_{xh}\}_{x,h=1}^{n',m_j}$
(i.e., the rows of $U$ correspond to the coefficients of the sought vectors).
Assume that $\F$ is a subfield of the field of real numbers $\RR$
such that the dual eigenmatrix $Q$ of $\A$ has entries from $\F$
(i.e., $Q \in \F^{\I \times \J}$).
We may build a matrix $C \in \F^{Y \times Y}$
such that $C_{xy} = {Q_{ij} \over m_j}$ holds
whenever $(x, y) \in R'_i$ ($x, y \in Y$, $i \in \I'$),
and pass it as an input to Algorithm~\ref{alg:unitvecs}
along with the chosen index $j \in \J$.
If the algorithm succeeds,
it computes all the entries in the matrix $U$,
thus giving an embedding of $\A'$ into $S_j$.
On the other hand, if the algorithm fails,
we may conclude that no such embedding exists.

\begin{algorithm}[t]
\begin{algorithmic}
\Require $j \in \J$, $C \in \F^{Y \times Y}$ such that
$C_{xy} = {Q_{ij} \over m_j} \Longleftrightarrow (x, y) \in R'_i$
($x, y \in Y$, $i \in \I'$)
\For{$x = 1, 2, \dots, n'$}
\Comment{$Y = \{1, 2, \dots, n'\}$}
	\State $h \gets 1$
	\For{$y = 1, 2, \dots, x-1$}
		\State $d \gets C_{xy} - \sum_{k=1}^{h-1} a_{xk} a_{yk}$
		\If{$h \le m_j \land a_{yh} \ne 0$}
			\State $a_{xh} \gets {d \over a_{yh}}$
			\State $h \gets h + 1$
		\ElsIf{$d \ne 0$}
			\State {\bf fail}
			\Comment{Cannot obtain the inner products}
		\EndIf
	\EndFor
	\State $s \gets \sum_{k=1}^{h-1} a_{xk}^2$
	\If{$s > 1$}
		\State {\bf fail}
		\Comment{The norm is larger than one}
	\ElsIf{$s < 1$}
		\If{$h > m_j$}
			\State {\bf fail}
			\Comment{The norm is smaller than one}
		\EndIf
		\State $a_{xh} \gets \sqrt{1 - s}$
		\State $h \gets h + 1$
	\EndIf
	\For{$k = h, h+1, \dots, m_j$}
		\State $a_{xk} \gets 0$
	\EndFor
\EndFor
\end{algorithmic}

\caption{The algorithm for computing the coefficients
of the unit vectors $u'_x$ ($x \in Y$) in an orthonormal basis of $S_j$.}
\label{alg:unitvecs}
\end{algorithm}

For an element $\beta \in \F$ such that $\beta > 0$,
we define the set
$\F \sqrt{\beta} = \{\alpha \sqrt{\beta} \mid \alpha \in \F\}$,
where $\sqrt{\beta}$ is the unique positive real number
such that $(\sqrt{\beta})^2 = \beta$.
Furthermore,
we define the set
$\F \sqrt{\F} = \bigcup_{\substack{\beta \in \F \\ \beta > 0}} \F\sqrt{\beta}$.
We note that the sets $\F \sqrt{\beta}$ are closed under addition,
and for $\gamma, \delta \in \F \sqrt{\beta}$,
we have $\gamma \delta \in \F$;
similarly, for $\alpha \in \F$, $\gamma \in \F \sqrt{\beta} \setminus \{0\}$,
we have ${\alpha \over \gamma} \in \F \sqrt{\beta}$.
This implies that there exist numbers
$\beta_h \in \F$, $\beta_h > 0$ ($1 \le h \le m_j$)
such that $a_{xh} \in \F \sqrt{\beta_h}$ for all $x \in Y$
-- i.e.,
all the entries of the $h$-th column of $U$ are elements of $\F \sqrt{\beta_h}$.
Therefore, we have $U \in (\F \sqrt{\F})^{Y \times m_j}$.

The {\tt eigenspace-embeddings} repository~\cite{v25}
contains an implementation of Algorithm~\ref{alg:unitvecs}
based on SageMath~\cite{sage24}.
We use the {\tt sage-drg} package~\cite{v18a,v19}
to compute the dual eigenmatrix $Q$
of an association scheme $\A$ with the given parameters.
The package has been adapted so that the computed parameters
(assuming they do not depend on a variable)
are returned in SageMath's implementation of the rational field $\Q$,
provided by the object {\tt QQ} (an element of class {\tt RationalField}),
or a minimal extension thereof
(an element of the class {\tt NumberField}).
In both cases, SageMath's implementation is based on PARI~\cite{pari23}.
Thus, $\F$ is a (possibly trivial) extension of $\Q$.
For the computation of the entries of $U$ in Algorithm~\ref{alg:unitvecs},
we implement a class {\tt IncompleteSqrtExtension}
to provide the required arithmetic in the pseudo-field $\F \sqrt{\F}$.
Note that the latter set is closed under multiplication,
but not under addition, and is therefore not a field,
yet it is implemented as a subclass of {\tt NumberField},
with addition of elements of $\F \sqrt{\F}$
not belonging to a common subset $\F \sqrt{\beta}$ triggering an error
(note that this cannot happen in Algorithm~\ref{alg:unitvecs}).
Such an approach has a great performance and correctness advantage
over using the symbolic ring,
provided by SageMath's object {\tt SR}
(which is still used by {\tt sage-drg} in the presence of variables),
as the generality of the latter means that
the obtained expressions often cannot be adequately simplified,
thus leading to bad performance and incorrect results.

Since the existence of an embedding into the eigenspace $S_j$
only depends on the parameters of $\A$ and not its structure,
we may use the method described above to check for feasibility
of parameters of association schemes
and possibly attempt to find new constructions or characterizations
for a given parameter set.
In particular,
this method will prove to be useful
when some substructure of the association scheme is already known,
as we can then build on this substructure
and explore which possibilities are admissible
until either a contradiction occurs,
or the desired characterization or construction has been reached.

Suppose that $\A$ is an imprimitive $d$-class association scheme
with a nontrivial imprimitivity set $\tilde{0}$.
Then,
for every equivalence class $Y$ of $R_{\tilde{0}}$,
$\A|_Y$ is a $d'$-class association scheme on $n'$ vertices,
where $d' = |\tilde{0}| - 1 < d$ and $n' < n$.
We thus obtain smaller association schemes on subsets of vertices of $\A$,
and their parameters are determined by the parameters of $\A$.
Even when $\A$ is only specified by its parameters
and its precise structure is not known,
the subschemes on these subsets
might be determined or characterized by the parameters,
allowing further consideration by the above method.

Let $\overline{0}$ be the dual imprimitivity set
corresponding to the imprimitivity set $\tilde{0}$,
and let $\{X_\ell \mid \ell = 1, 2, \dots, \tilde{n}\}$ be the set
of the equivalence classes of $R_{\tilde{0}}$.
As
${Q_{ij} \over m_j} = {\overline{Q}_{i \overline{\jmath}} \over \overline{m}_{\overline{\jmath}}}$
holds for all $i \in \tilde{0}$, $\overline{\jmath} \in \overline{\J}$
and $j \in \overline{\jmath}$,
we see that an embedding of $\A|_{X_\ell}$ ($1 \le \ell \le \tilde{n}$)
into $S_{\overline{\jmath}}|_{X_\ell}$ can be naturally extended
to an embedding into $S_j$ for each $j \in \overline{\jmath}$.
In particular,
when the imprimitivity set is of the form $\tilde{0} = \{0, i^*\}$,
the graph $(X, R_{i^*})$ is isomorphic to a union of $(k_{i^*}+1)$-cliques,
and we call the sets $X_\ell$ ($1 \le \ell \le \tilde{n}$)
the {\em $R_{i^*}$-cliques},
and the embeddings of their vertices
into $S_j$ $(j \in \J \setminus \overline{0})$
correspond to the vertices of a $k_{i^*}$-dimensional regular simplex,
thus spanning a $k_{i^*}$-dimensional subspace of $S_j$.
Generalizing this to the case when $|\tilde{0}| > 2$,
we may call the sets $X_\ell$ ($1 \le \ell \le \tilde{n}$)
the {\em $R$-cliques},
where $R = \bigcup_{i \in \tilde{0} \setminus \{0\}} R_i$.

For several parameter sets,
we will determine the possible induced subschemes of $\A$
on a small number of $R$-cliques (or subsets thereof).
If we manage to determine that none of these possibilities
admit an embedding into an eigenspace of $\A$,
we then conclude that such an association scheme does not exist.
To this end,
we will consider eigenspaces $S_j$ with $\overline{m}_{\overline{\jmath}} > 1$
of small dimension.
In particular,
we will consider cases
when ${m_j \over \overline{m}_{\overline{\jmath}}} \le 3$.
Once we manage to construct a set $Y \subseteq X$
such that the vectors $\{u'_x \mid x \in Y\}$ span the eigenspace $S_j$,
we may use the intersection numbers of $\A$
to determine the possible choices for the relations $R_{xy} \in \R$ ($x \in Y$)
such that $(x, y) \in R_{xy}$ for a candidate vertex $y \not\in Y$,
and examine which of the corresponding vectors $u'_y$ are unit vectors.
We may then try to find a subset $Z$ of these vectors such that $|Y| + |Z| = n$
and for each pair of vertices $y, z \in Z$
we have $\langle u'_y, u'_z \rangle = {Q_{ij} \over m_j}$ for some $i \in \I$
(i.e., $(y, z) \in R_i$).
Finally,
we may verify that $(Y \cup Z, \R)$ is indeed an association scheme
with the parameters of $\A$.
Alternatively,
if no such set $Z$ can be found for any of the choices of $Y$
such that the association scheme $\A$
necessarily contains a subscheme isomorphic to $\A|_Y$,
we conclude that $\A$ does not exist.

In particular, given an association scheme $\A = (X, \R)$,
we will define the vertex subsets
$X^{(t)} = \bigcup_{\ell=1}^t X_\ell$ ($1 \le t \le \tilde{n}$),
the induced subschemes
$\A^{(t)} = \A|_{X^{(t)}} = (X^{(t)}, \{R_i^{(t)} \mid i \in \I\} \setminus \{\emptyset\})$
with $R_i^{(t)} = R_i|_{X^{(t)}}$ ($i \in \I$),
and the graphs $\G^{(t)}_i = (X^{(t)}, R_i^{(t)})$
($i \in \I$, $R_i \ne \Id_X$).
We will consider the candidate relation schemes for $\A^{(t)}$
for certain choices of $t$
by considering the possible choices of the graphs $\G^{(t)}_i$ ($i \in \I$),
as well as some other induced subschemes and their corresponding graphs,
and attemtpt to find the embeddings of these relation schemes
into an eigenspace $S_j$ of $\A$ for some $j \in \J \setminus \overline{0}$.

\subsection{A small example}
\label{ssec:example}

We will now apply the method described above on a small example.
Consider a formally self-dual $3$-class association scheme $\A$ on $8$ vertices
with index sets $\I = \J = \{0, 1, 2, 3\}$
given by its eigenmatrices
$$
P = Q = \begin{pmatrix}
1 &  3 &  3 &  1 \\
1 &  1 & -1 & -1 \\
1 & -1 & -1 &  1 \\
1 & -3 &  3 & -1
\end{pmatrix}.
$$
The association scheme $\A$ is imprimitive
with imprimitivity set $\tilde{0} = \{0, 2\}$
and the corresponding dual imprimitivity set $\overline{0} = \{0, 3\}$.
Since the scheme is formally self-dual,
the roles of these two sets can also be reversed.
We note that $R_{\tilde{0}}$
has two equivalence classes $X_1$ and $X_2$ of size $4$
-- they are the $R_2$-cliques of $\A$.

We will consider the embedding of $\A^{(1)} = \A|_{X_1}$
into the eigenspace $S_1$ of dimension $m_1 = 3$.
We use Algorithm~\ref{alg:unitvecs} to compute the matrix $U$
with the coefficients of the unit vectors $u'_x \in S_1$ ($x \in X_1$).
$$
U = \begin{pmatrix}
1 &  0 &  0 \\
-{1 \over 3} & {2 \sqrt{2} \over 3} & 0 \\
-{1 \over 3} & -{\sqrt{2} \over 3} & {\sqrt{6} \over 3} \\
-{1 \over 3} & -{\sqrt{2} \over 3} & -{\sqrt{6} \over 3}
\end{pmatrix}
$$
Since the above matrix has full column rank,
we may consider the candidates for the remaining four vertices
in the $R_2$-clique $X_2$.
As we have $k_1 = 3$ and $k_3 = 1$,
it follows that each vertex of $X_2$
is in relation $R_3$ with precisely one vertex of $X_1$
and in relation $R_1$ with the remaining three vertices of $X_1$.
We thus have precisely four candidates for the vertices $y \in X_2$,
and we find that the coefficients
of the correspoding unit vectors $u'_y \in S_1$
are given by the matrix $-U$.
By considering the inner products
between the rows of the matrices $U$ and $-U$,
we build the relation matrix $R$ of the obtained relation scheme.
$$
R = \begin{pmatrix}
0 & 2 & 2 & 2 & 3 & 1 & 1 & 1 \\
2 & 0 & 2 & 2 & 1 & 3 & 1 & 1 \\
2 & 2 & 0 & 2 & 1 & 1 & 3 & 1 \\
2 & 2 & 2 & 0 & 1 & 1 & 1 & 3 \\
3 & 1 & 1 & 1 & 0 & 2 & 2 & 2 \\
1 & 3 & 1 & 1 & 2 & 0 & 2 & 2 \\
1 & 1 & 3 & 1 & 2 & 2 & 0 & 2 \\
1 & 1 & 1 & 3 & 2 & 2 & 2 & 0
\end{pmatrix}
$$
It can be easily verified that $R$ is the relation matrix
of the association scheme corresponding to the $3$-cube $Q_3$.
This confirms the well-known fact that this is the unique association scheme
with the parameters given above.

\section{Nonexistence results}
\label{sec:nonex}

We will now attempt to use the technique
described in Section~\ref{sec:embeddings}
to study the parameter sets marked as feasible in~\cite{hm23a}
which pass all known feasibility conditions
(see Appendix~\ref{app:known} for those that do not).
We find three parameter sets for which we show nonexistence,
of which two correspond to imprimitive $4$-class association schemes
and one corresponds to an imprimitive $5$-class association scheme.

Besides the software and algorithms mentioned in Section~\ref{sec:embeddings},
we also use the following software
which is included in the SageMath computer algebra system~\cite{sage24}:
{\tt nauty}~\cite{m90} for graph generation,
{\tt bliss}~\cite{jk07a,bliss21} for automorphism group computation,
GAP~\cite{g24} for group manipulation,
and GLPK~\cite{glpk20} for solving integer linear programs
(for graph coloring).

\subsection{QPG with parameter array $[[12, 4, 4, 24], [6, 0, 3; 0, 1; 2]]$}
\label{ssec:p4-12-45-52}

Let $\A$ be a $4$-class association scheme with intersection numbers
\begin{equation}
\label{eqn:p4-12-45-52}
\begin{aligned}
(p^0_{ij})_{i,j=0}^4 &= \begin{pmatrix}
1 &  0 & 0 & 0 &  0 \\
0 & \mathit{12} & 0 & 0 &  0 \\
0 &  0 & \mathit{4} & 0 &  0 \\
0 &  0 & 0 & \mathit{4} &  0 \\
0 &  0 & 0 & 0 & \mathit{24}
\end{pmatrix}, &
(p^1_{ij})_{i,j=0}^4 &= \begin{pmatrix}
0 & 1 & 0 & 0 &  0 \\
1 & 3 & 2 & 0 &  6 \\
0 & 2 & 0 & 0 &  2 \\
0 & 0 & 0 & 0 &  4 \\
0 & 6 & 2 & 4 & 12
\end{pmatrix}, \\
(p^2_{ij})_{i,j=0}^4 &= \begin{pmatrix}
0 & 0 & 1 & 0 &  0 \\
0 & \mathit{6} & 0 & 0 &  6 \\
1 & 0 & 1 & 2 &  0 \\
0 & 0 & 2 & 2 &  0 \\
0 & 6 & 0 & 0 & 18
\end{pmatrix}, &
(p^3_{ij})_{i,j=0}^4 &= \begin{pmatrix}
0 &  0 & 0 & 1 &  0 \\
0 &  \mathit{0} & \mathit{0} & 0 & 12 \\
0 &  \mathit{0} & 2 & 2 &  0 \\
1 &  0 & 2 & 1 &  0 \\
0 & 12 & 0 & 0 & 12
\end{pmatrix}, \\
(p^4_{ij})_{i,j=0}^4 &= \begin{pmatrix}
0 & 0 & 0 & 0 &  1 \\
0 & \mathit{3} & \mathit{1} & \mathit{2} &  6 \\
0 & \mathit{1} & 0 & 0 &  3 \\
0 & \mathit{2} & 0 & 0 &  2 \\
1 & 6 & 3 & 2 & 12
\end{pmatrix}.
\end{aligned}
\end{equation}
Note that each of the above matrices $(p^h_{ij})_{i,j=0}^4$ ($0 \le h \le 4$)
shows the number of vertices in relations $R_i$ and $R_j$
from a pair of vertices in relation $R_h$,
or, equivalently,
the coefficients of $A_h$ in the product $A_i A_j$.
The intersection numbers appearing in the parameter array are shown in italics
-- note that due to $\A$ being a symmetric association scheme,
the matrices above are symmetric
and some of these parameters are thus marked twice.
The same notation will also be used in the remaining examples.

The graph $\G_1 = (X, R_1)$ is a quotient-polynomial graph on $45$ vertices
with parameter array $[[12, 4, 4, 24],$ $[6, 0, 3; 0, 1; 2]]$.
The association scheme $\A$ is imprimitive
with imprimitivity set $\tilde{0} = \{0, 2, 3\}$.
The dual eigenmatrix of $\A$ is
$$
Q = \begin{pmatrix}
1 & 10 & 20 &  4 & 10 \\
1 & \frac{5 \sqrt{2}}{2} &  0 & -1 & -\frac{5 \sqrt{2}}{2} \\
1 & \frac{5}{2} & -10 &  4 & \frac{5}{2} \\
1 & -5 &  5 &  4 & -5 \\
1 & -\frac{5 \sqrt{2}}{4} & 0 & -1 & \frac{5 \sqrt{2}}{4}
\end{pmatrix}.
$$
By the ordering of eigenspaces used in the above matrix,
the corresponding dual imprimitivity set is $\overline{0} = \{0, 3\}$,
and we also have $\overline{1} = \overline{4} = \{1, 4\}$
and $\overline{2} = \{2\}$.
Let $\{X_\ell \mid 1 \le \ell \le 5\}$
be the set of the equivalence classes of $R_{\tilde{0}}$,
and note that the graphs $\G_2|_{X_\ell} = (X_\ell, R_2|_{X_\ell})$
($1 \le \ell \le 5$)
are isomorphic to the graph $K_3 \square K_3$,
or, equivalently, the Hamming graph $H(2, 3)$.
We will call their maximal cliques (of size $3$) {\em lines}
-- taking the vertices as points,
this gives us a geometry of generalized quadrangle $\GQ(2, 1)$.
The sets $X_\ell$ ($1 \le \ell \le 5$) are thus
the $(R_2 \cup R_3)$-cliques of $\A$,
and we also have
$X_\ell = Y_{\ell 1} \cup Y_{\ell 2} \cup Y_{\ell 3} = Z_{\ell 1} \cup Z_{\ell 2} \cup Z_{\ell 3}$,
where $Y_{\ell r}$ and $Z_{\ell s}$ ($1 \le \ell \le 5$, $1 \le r, s \le 3$)
are the lines of $\G_2|_{X_\ell}$
such that $|Y_{\ell r} \cap Z_{\ell s}| = 1$.
In particular,
$\{Y_{\ell 1}, Y_{\ell 2}, Y_{\ell 3}\}$
and $\{Z_{\ell 1}, Z_{\ell 2}, Z_{\ell 3}\}$
are partitions of $X_\ell$ into disjoint lines
-- we call these partitions the {\em spreads} of $\G_2|_{X_\ell}$.

We will consider embeddings
of subschemes induced on three $(R_2 \cup R_3)$-cliques of $\A$
into the eigenspace $S_1$ of dimension $m_1 = 10$.
We note that $\overline{m}_{\overline{1}} = 4$
and therefore ${m_1 \over \overline{m}_{\overline{1}}} = {5 \over 2} < 3$,
which may severely restrict
which of such subschemes admit an embedding into $S_1$.
Since we will encounter a similar situation later,
let us first prove the following lemma.

\begin{lemma}
\label{lem:hamming}
Let $\A = (X, \R = \{R_i \mid i \in \I\})$ be an association scheme
with imprimitivity set $\tilde{0}$
and intersection numbers $p^j_{ij} = e-1$ and $p^j_{i'j} = 0$
for some $i \in \tilde{0}$, $j \in \I \setminus \tilde{0}$,
and all $i' \in \tilde{0} \setminus \{0, i\}$.
Suppose that $Y$ and $Y'$ are equivalence classes of $R_{\tilde{0}}$
such that the graphs $\G_i|_Y$ and $\G_i|_{Y'}$
are both isomorphic to the Hamming graph $H(d, e)$ for some $d < e$,
and for all $x \in Y$, $y \in Y'$,
$(x, y) \in R_{j'}$ holds for some $j' \in \tilde{\jmath}$.
Then the graph $\G_j|_{Y \cup Y'}$ is isomorphic to $e^{d-1} K_{e,e}$,
and the partitions of $Y$ and $Y'$
corresponding to the connected components of $\G_j|_{Y \cup Y'}$
coincide with spreads of $\G_i|_Y$ and $\G_i|_{Y'}$
(i.e., they are partitioned into maximal $R_i$-cliques of size $e$).
\end{lemma}

\begin{proof}
Since $\sum_{i' \in \tilde{0}} p^j_{i'j} = e$,
the graph $\G_j|_{Y \cup Y'}$ is $e$-regular;
since $Y$ and $Y'$ are equivalence classes of $R_{\tilde{0}}$,
it is also bipartite with bipartition $Y + Y'$.
Let $x$ be a vertex from $Y$,
and denote the set of neighbours of $x$ by $L'$.
As $p^j_{i'j} = 0$ for all $i' \in \tilde{0} \setminus \{0, i\}$,
the set $L' \subseteq Y'$ forms a maximal $R_i$-clique.
Furthermore,
the set of neighbours of each vertex from $L'$
forms a maximal $R_i$-clique inside $Y$ containing $x$.
Since $x$ is contained in precisely $d$ maximal $R_i$-cliques,
which is less than $|L'| = e$,
by the pigeonhole principle,
there exist two distinct vertices $y, y' \in L'$
with the same set of neighbours,
which we denote by $L$.
Then,
every vertex of $L$ is adjacent to all vertices
in the unique maximal $R_i$-clique containing $y$ and $y'$
-- i.e.,
$\G_j|_{L \cup L'}$ is a connected component of $\G_j|_{Y \cup Y'}$
and is isomorphic to $K_{e,e}$.
Since $x$ was arbitrary,
it follows that $\G_j|_{Y \cup Y'}$ consists of $e^{d-1}$
connected components isomorphic to $K_{e,e}$
whose bipartitions consist of maximal $R_i$-cliques in $Y$ and $Y'$.
\end{proof}

We may now give the following result.

\begin{theorem}
\label{thm:p4-12-45-52}
An association scheme $\A$ with intersection numbers \eqref{eqn:p4-12-45-52}
does not exist.
\end{theorem}

\begin{proof}
Since $p^1_{12} = 2$ and $p^1_{13} = 0$,
we may apply Lemma~\ref{lem:hamming} to conclude that
for each choice of $(R_2 \cup R_3)$-cliques $X_\ell$ and $X_{\ell'}$ of $\A$
($1 \le \ell < \ell' \le 5$),
the graph $\G_1|_{X_\ell \cup X_{\ell'}}$ is isomorphic to $3K_{3,3}$,
with the partitions of vertices of $X_\ell$ and $X_{\ell'}$
corresponding to the connected components of $\G_1|_{X_\ell \cup X_{\ell'}}$
coincinding with spreads of $\G_2|_{X_\ell}$ and $\G_2|_{X_{\ell'}}$.

We will now consider the possibilities for the relation scheme $\A^{(3)}$.
There are six mutually non-isomorphic possibilities for the graph $\G^{(3)}_1$:
either the same spread is used in each of $\G_1|_{X_\ell}$ ($1 \le \ell \le 3$)
to determine the edges to the other two $(R_2 \cup R_3)$-cliques,
and $\G^{(3)}_1$ has one, two or three connected components,
or different spreads are used for one, two or three
of $\G_1|_{X_\ell}$ ($1 \le \ell \le 3$),
see Figure~\ref{fig:a3}.
The choice of this graph
thus uniquely determines the relation scheme $\A^{(3)}$.

\begin{figure}[t]
\makebox[\textwidth][c]{
	\beginpgfgraphicnamed{fig-a3}
	\begin{tikzpicture}[style=thick,scale=0.7]
		\input{a3.tikz}
	\end{tikzpicture}
	\endpgfgraphicnamed
}

\caption{The candidates for the relation scheme $\A^{(3)}$.
In each case,
the lines are represented with rounded rectangles
and form three $\GQ(2, 1)$ geometries;
the vertices are implied at intersections of lines.
Two distinct vertices are in relation $R^{(3)}_1$
if they are contained in two lines connected by an edge,
in relation $R^{(3)}_2$ if they are contained in a common line,
in relation $R^{(3)}_3$ if they are contained in distinct lines
of the same $\GQ(2, 1)$,
and in relation $R^{(3)}_4$ otherwise.
}
\label{fig:a3}
\end{figure}
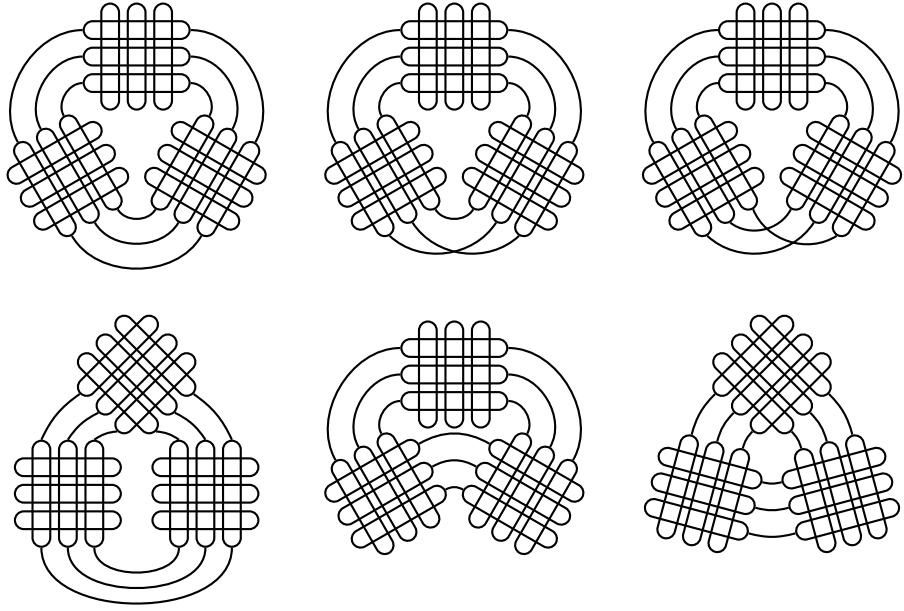

For each of these possibilities,
we thus build a candidate for the relation scheme $\A^{(3)}$
and attempt to compute the corresponding matrix $U$
with the coefficients of the unit vectors $u'_x \in S_1$ ($x \in X^{(3)}$)
using Algorithm~\ref{alg:unitvecs}.
We find all the coefficients can only be determined
in the case when different spreads are used
to determine the edges from two of the three $(R_2 \cup R_3)$-cliques,
and thus conclude that the relation schemes corresponding to the other cases
do not admit an embedding into $S_1$.

Let us consider the sole remaining possibility.
Without loss of generality,
we may assume that the connected components
of $\G_1|_{X_1 \cup X_2}$, $\G_1|_{X_1 \cup X_3}$ and $\G_1|_{X_2 \cup X_3}$
are $Y_{1r} \cup Y_{2r}$, $Y_{1r} \cup Y_{3r}$ and $Z_{2s} \cup Z_{3s}$
($1 \le r, s \le 3$),
respectively.
The corresponding matrix of coefficients
has full column rank,
thus uniquely determining the orthonormal basis of $S_1$ being used.

By the argument above,
each vertex $y \in X_4 \cup X_5$ must be in relation $R_1$
precisely with all vertices of one line within each of $X_1, X_2, X_3$.
Since there are $6$ lines in each of $\G_2|_{X_\ell}$ ($1 \le \ell \le 3$),
we examine the $6^3 = 216$ candidates for such vertices $y$
and attempt to find the corresponding unit vectors $u'_y$.
However, we find that no such unit vectors exist,
from which it follows that the association scheme $\A$ does not exist.
\end{proof}

The
\href{https://nbviewer.org/github/jaanos/eigenspace-embeddings/blob/main/QPG4-12-45-52.ipynb}{\tt QPG4-12-45-52.ipynb}
notebook on the {\tt eigenspace-embeddings} repository~\cite{v25}
illustrates the computation needed to obtain the above result.

\subsection{QPG with parameter array $[[8, 8, 4, 24], [1, 0, 2; 2, 1; 1]]$}
\label{ssec:p4-8-45-18}

Let $\A$ be a $4$-class association scheme with intersection numbers
\begin{equation}
\label{eqn:p4-8-45-18}
\begin{aligned}
(p^0_{ij})_{i,j=0}^4 &= \begin{pmatrix}
1 & 0 & 0 & 0 &  0 \\
0 & \mathit{8} & 0 & 0 &  0 \\
0 & 0 & \mathit{8} & 0 &  0 \\
0 & 0 & 0 & \mathit{4} &  0 \\
0 & 0 & 0 & 0 & \mathit{24}
\end{pmatrix}, &
(p^1_{ij})_{i,j=0}^4 &= \begin{pmatrix}
0 & 1 & 0 & 0 &  0 \\
1 & 0 & 1 & 0 &  6 \\
0 & 1 & 3 & 1 &  3 \\
0 & 0 & 1 & 0 &  3 \\
0 & 6 & 3 & 3 & 12
\end{pmatrix}, \\
(p^2_{ij})_{i,j=0}^4 &= \begin{pmatrix}
0 & 0 & 1 & 0 &  0 \\
0 & \mathit{1} & 3 & 1 &  3 \\
1 & 3 & 1 & 0 &  3 \\
0 & 1 & 0 & 0 &  3 \\
0 & 3 & 3 & 3 & 15
\end{pmatrix}, &
(p^3_{ij})_{i,j=0}^4 &= \begin{pmatrix}
0 & 0 & 0 & 1 &  0 \\
0 & \mathit{0} & \mathit{2} & 0 &  6 \\
0 & \mathit{2} & 0 & 0 &  6 \\
1 & 0 & 0 & 3 &  0 \\
0 & 6 & 6 & 0 & 12
\end{pmatrix}, \\
(p^4_{ij})_{i,j=0}^4 &= \begin{pmatrix}
0 & 0 & 0 & 0 &  1 \\
0 & \mathit{2} & \mathit{1} & \mathit{1} &  4 \\
0 & \mathit{1} & 1 & 1 &  5 \\
0 & \mathit{1} & 1 & 0 &  2 \\
1 & 4 & 5 & 2 & 12
\end{pmatrix}.
\end{aligned}
\end{equation}
The graph $\G_1 = (X, R_1)$ is a quotient-polynomial graph on $45$ vertices
with parameter array $[[12, 4, 4, 24],$ $[6, 0, 3; 0, 1; 2]]$.
The association scheme $\A$ is imprimitive
with imprimitivity set $\tilde{0} = \{0, 3\}$.
This parameter set is also listed in~\cite{hm23} as a feasible parameter set
for an association scheme with noncyclotomic eigenvalues,
as the dual eigenmatrix of $\A$ has entries from a degree $6$ extension of $\Q$.
Therefore, we only list their first three decimal places:
$$
Q = \begin{pmatrix}
1 & 12     &  12    &  8 & 12     \\
1 &  3.829 &  2.086 & -1 & -5.915 \\
1 & -5.430 &  4.925 & -1 &  0.505 \\
1 & -3     & -3     &  8 & -3     \\
1 &  0.534 & -2.337 & -1 &  1.803
\end{pmatrix}.
$$
By the ordering of eigenspaces used in the above matrix,
the corresponding dual imprimitivity set is $\overline{0} = \{0, 3\}$,
and we also have $\overline{1} = \{1, 2, 4\}$.
We will consider embeddings
of subschemes induced on three $R_3$-cliques of $\A$
into the eigenspace $S_1$ of dimension $m_1 = 12$.
We note that $\overline{m}_{\overline{1}} = 4$
and therefore ${m_1 \over \overline{m}_{\overline{1}}} = 3$,
which may severely restrict
which of such subschemes admit an embedding into $S_1$.

Let $\alpha_i = {Q_{i1} \over m_1}$ ($0 \le i \le 4$)
be the inner product of the unit vectors in $S_1$
corresponding to two vertices in relation $R_i$.
We note that $\alpha_0 = 1$, $\alpha_3 = -{1 \over 4}$,
$\alpha_1 + \alpha_2 + 3 \alpha_4 = 0$
and $\alpha_2 = 8 \alpha_1^2 + 2 \alpha_1 - 1$.
Since the minimal polynomial of $\alpha_1$ is
$x^3 - {3 \over 16} x + {7 \over 256}$,
it follows that $\alpha_i \in \F$ ($0 \le i \le 4$),
where $\F$ is a degree $3$ extension of $\Q$.
We may therefore use the field $\F$
when computing coefficients of vectors in an orthonormal basis of $S_1$
using Algorithm~\ref{alg:unitvecs}.
This allows us to obtain the following result.

\begin{theorem}
\label{thm:p4-8-45-18}
An association scheme $\A$ with intersection numbers \eqref{eqn:p4-8-45-18}
does not exist.
\end{theorem}

\begin{proof}
Since $p^1_{31} + 1 = p^2_{31} = p^4_{31} = 1$
and $p^1_{32} = p^2_{32} + 1 = p^4_{32} = 1$,
each vertex $x$ of $\A$ is in relations $R_1$ and $R_2$
with precisely one vertex in each of the $R_3$-cliques not containing $x$.
As $p^1_{11} = 0$ and $p^2_{22} = 1$,
the graph $\G_1$ does not contain triangles,
while the graph $\G_2$ does contain a triangle $x_1 x_2 x_3$.
Without loss of generality,
we may assume $x_\ell \in X_\ell$ ($1 \le \ell \le 3$).
We will consider the possibilities for the relation scheme $\A^{(3)}$
under this labelling of the $R_3$-cliques.

By the above argument,
the graphs $\G^{(3)}_1$ and $\G^{(3)}_2$
are unions of cycles of lengths divisible by $3$.
Since $\G^{(3)}_1$ has no triangles,
it must be isomorphic to either $C_{15}$ or $C_9 + C_6$.
On the other hand, by the above assumption,
the graph $\G^{(3)}_2$ does contain a triangle.
The choice of these two graphs
uniquely determines the relation scheme $\A^{(3)}$.

Given a choice of $\G^{(3)}_1$,
we define an asymmetric relation
$\vec{R}^{(3)}_{2,4} = \{(x, y) \in R^{(3)}_2 \cup R^{(3)}_4 \mid x \in X_i, \ y \in X_j, \ j - i \equiv 1 \pmod{3}\}$
and a directed graph
$\vec{\G}^{(3)}_{2,4} = (X^{(3)}, \vec{R}^{(3)}_{2,4})$.
The candidates for $\G^{(3)}_2$ are then precisely
the underlying graphs of $1$-factors of $\vec{\G}^{(3)}_{2,4}$
(i.e., spanning subgraphs with in- and out-degrees of all vertices equal to $1$).
In the cases when $\G^{(3)}_1$ is isomorphic to $C_{15}$ and $C_9 + C_6$,
we find $5704$ and $4736$ distinct (up to isomorphism)
$1$-factors of $\vec{\G}^{(3)}_{2,4}$, respectively,
of which $3637$ and $3028$, respectively, contain a triangle.
For each of these possibilities,
we thus build a candidate for the relation scheme $\A^{(3)}$
and attempt to compute the corresponding matrix $U$
with the coefficients of the unit vectors $u'_x \in S_1$ ($x \in X^{(3)}$)
using Algorithm~\ref{alg:unitvecs}.
We find all the coefficients can only be determined
in $55$ cases with $\G^{(3)}_1 \cong C_{15}$
and $45$ cases with $\G^{(3)}_1 \cong C_9 + C_6$.
In each of these cases,
the matrix $U$ has full column rank,
thus uniquely determining the orthornomal basis of $S_1$ being used.

For each of these $100$ cases,
we examine the $(5 \cdot 4)^3 = 8000$ candidates
for the remaining vertices $y$ of $\A$
(determined by the choice of vertices in relations $R_1$ and $R_2$ with $y$
in each of the $R_3$-cliques $X_1, X_2, X_3$)
and attempt to find the corresponding unit vectors $u'_y$.
However,
in none of the cases we find any such unit vectors,
from which it follows that the association scheme $\A$ does not exist.
\end{proof}

The
\href{https://nbviewer.org/github/jaanos/eigenspace-embeddings/blob/main/QPG4-8-45-18.ipynb}{\tt QPG4-8-45-18.ipynb}
notebook on the {\tt eigenspace-embeddings} repository~\cite{v25}
illustrates the computation needed to obtain the above result.

\subsection{QPG with parameter array
$[[6, 18, 2, 6, 12], [1, 0, 2, 0; 0, 0, 3; 0, 1; 2]]$}
\label{ssec:p5-6-45-22}

Let $\A$ be a $5$-class association scheme with intersection numbers
\begin{equation}
\label{eqn:p5-6-45-22}
\begin{aligned}
(p^0_{ij})_{i,j=0}^5 &= \begin{pmatrix}
1 & 0 &  0 & 0 & 0 &  0 \\
0 & \mathit{6} &  0 & 0 & 0 &  0 \\
0 & 0 & \mathit{18} & 0 & 0 &  0 \\
0 & 0 &  0 & \mathit{2} & 0 &  0 \\
0 & 0 &  0 & 0 & \mathit{6} &  0 \\
0 & 0 &  0 & 0 & 0 & \mathit{12}
\end{pmatrix}, &
(p^1_{ij})_{i,j=0}^5 &= \begin{pmatrix}
0 & 1 & 0 & 0 & 0 & 0 \\
1 & 0 & 3 & 0 & 2 & 0 \\
0 & 3 & 9 & 0 & 0 & 6 \\
0 & 0 & 0 & 0 & 0 & 2 \\
0 & 2 & 0 & 0 & 0 & 4 \\
0 & 0 & 6 & 2 & 4 & 0
\end{pmatrix}, \\
(p^2_{ij})_{i,j=0}^5 &= \begin{pmatrix}
0 & 0 & 1 & 0 & 0 & 0 \\
0 & \mathit{1} & 3 & 0 & 0 & 2 \\
1 & 3 & 0 & 2 & 6 & 6 \\
0 & 0 & 2 & 0 & 0 & 0 \\
0 & 0 & 6 & 0 & 0 & 0 \\
0 & 2 & 6 & 0 & 0 & 4
\end{pmatrix}, &
(p^3_{ij})_{i,j=0}^5 &= \begin{pmatrix}
0 & 0 &  0 & 1 & 0 & 0 \\
0 & \mathit{0} &  \mathit{0} & 0 & 0 & 6 \\
0 & \mathit{0} & 18 & 0 & 0 & 0 \\
1 & 0 &  0 & 1 & 0 & 0 \\
0 & 0 &  0 & 0 & 6 & 0 \\
0 & 6 &  0 & 0 & 0 & 6
\end{pmatrix}, \\
(p^4_{ij})_{i,j=0}^5 &= \begin{pmatrix}
0 & 0 &  0 & 0 & 1 & 0 \\
0 & \mathit{2} &  \mathit{0} & \mathit{0} & 0 & 4 \\
0 & \mathit{0} & 18 & 0 & 0 & 0 \\
0 & \mathit{0} &  0 & 0 & 2 & 0 \\
1 & 0 &  0 & 2 & 3 & 0 \\
0 & 4 &  0 & 0 & 0 & 8
\end{pmatrix}, &
(p^5_{ij})_{i,j=0}^5 &= \begin{pmatrix}
0 & 0 & 0 & 0 & 0 & 1 \\
0 & \mathit{0} & \mathit{3} & \mathit{1} & \mathit{2} & 0 \\
0 & \mathit{3} & 9 & 0 & 0 & 6 \\
0 & \mathit{1} & 0 & 0 & 0 & 1 \\
0 & \mathit{2} & 0 & 0 & 0 & 4 \\
1 & 0 & 6 & 1 & 4 & 0
\end{pmatrix}.
\end{aligned}
\end{equation}
The graphs $\G_1 = (X, R_1)$ and $\G_5 = (X, R_5)$
are quotient-polynomial graphs on $45$ vertices
with parameter arrays $[[6, 18, 2, 6, 12], [1, 0, 2, 0; 0, 0, 3; 0, 1; 2]]$
and $[[12, 18, 2, 6, 6],$ $[4, 6, 8, 0; 0, 0, 6; 0, 2; 4]]$,
respectively.
The association scheme $\A$ is imprimitive
with imprimitivity set $\tilde{0} = \{0, 3, 4\}$.
The dual eigenmatrix of $\A$ is
$$
Q = \begin{pmatrix}
1 & 15 & 2 & 10 & 15 & 2 \\
1 & \frac{5 \sqrt{6}}{2} & \frac{\sqrt{5} - 1}{2} & 0 & -\frac{5 \sqrt{6}}{2} & -\frac{\sqrt{5} + 1}{2} \\
1 & 0 & -\frac{\sqrt{5} + 1}{2} & 0 & 0 & \frac{\sqrt{5} - 1}{2} \\
1 & -\frac{15}{2} & 2 & 10 & -\frac{15}{2} & 2 \\
1 & 0 & 2 & -5 & 0 & 2 \\
1 & -\frac{5 \sqrt{6}}{4} & \frac{\sqrt{5} - 1}{2} & 0 & \frac{5 \sqrt{6}}{4} & -\frac{\sqrt{5} + 1}{2}
\end{pmatrix}.
$$
By the ordering of eigenspaces used in the above matrix,
the corresponding dual imprimitivity set is $\overline{0} = \{0, 2, 5\}$.
We also have $\tilde{1} = \{1, 5\}$, $\tilde{2} = \{2\}$,
and $\overline{1} = \{1, 4\}$, $\overline{3} = \{3\}$.
Let $\{X_\ell \mid 1 \le \ell \le 5\}$
be the set of the equivalence classes of $R_{\tilde{0}}$,
and note that the graphs $\G_4|_{X_\ell} = (X_\ell, R_4|_{X_\ell})$
($1 \le \ell \le 5$)
are isomorphic to the complete tripartite graph $K_{3 \times 3}$.
We will therefore call the sets $X_\ell$ ($1 \le \ell \le 5$)
the {\em $R_4$-tricliques} of $\A$.
Furthermore, as $\{0, 3\}$ is another imprimitivity set of $\A$,
we have $X_\ell = X_{\ell 1} \cup X_{\ell 2} \cup X_{\ell 3}$
($1 \le \ell \le 5$),
where the sets $X_{\ell r}$ ($1 \le \ell \le 5$, $1 \le r \le 3$)
are the $R_3$-cliques of $\A$.
We also note that the quotient scheme $\tilde{\A} = \A / \tilde{0}$
is the cyclic scheme $C_5$
-- i.e.,
the graphs $\tilde{\G}_{\tilde{1}} = (\tilde{X}, \tilde{R}_{\tilde{1}})$
and $\tilde{\G}_{\tilde{2}} = (\tilde{X}, \tilde{R}_{\tilde{2}})$
are both isomorphic to the graph $C_5$.

We will consider embeddings
of subschemes induced on three $R_4$-tricliques of $\A$
into the eigenspace $S_1$ of dimension $m_1 = 15$.
We note that $\overline{m}_{\overline{1}} = 6$
and therefore ${m_1 \over \overline{m}_{\overline{1}}} = {5 \over 2} < 3$,
which may severely restrict
which of such subschemes admit an embedding into $S_1$.
We obtain the following result.

\begin{theorem}
\label{thm:p5-6-45-22}
An association scheme $\A$ with intersection numbers \eqref{eqn:p5-6-45-22}
does not exist.
\end{theorem}

\begin{proof}
Let $X_1, X_2, X_3$ be $R_4$-tricliques of $\A$
such that $(X_1, X_2) \in \tilde{R}_{\tilde{1}}$
and $(X_1, X_3),$ $(X_2, X_3) \in \tilde{R}_{\tilde{2}}$
(i.e., they correspond to a path of length $2$ in $\tilde{\G}_{\tilde{2}}$).
Let us first consider the graph $\G^{(2)}_1$.
Since $p^1_{13} = 0$ and $p^1_{14} = 2$,
each vertex of $X_\ell$ is in relations $R_1$
with precisely one vertex from each $R_3$-clique $X_{\ell' r}$
($\{\ell, \ell'\} = \{1, 2\}$, $1 \le r \le 3$).
The graph $\G^{(2)}_1$ is therefore a cubic bipartite graph on $18$ vertices
with bipartition $X_1 + X_2$
such that its distance-$2$ graph admits a $3$-coloring
(i.e., each color class partitions $X_i$ ($i = 1, 2$) into $R_3$-cliques).

We use the {\tt geng} utility from the {\tt nauty} package~\cite{m90}
to generate bipartite cubic graphs on $18$ vertices
and then use SageMath~\cite{sage24}
to pick the ones whose distance-$2$ graphs have chromatic number (at most) $3$.
We find $18$ such graphs,
all of which have at most one connected component not isomorphic to $K_{3,3}$.
Therefore,
each of these graphs has a unique bipartition up to graph automorphism,
which we identify with $X_1 + X_2$.
Furthermore, for each such graph $\G$,
we find all $6$-colorings of the graph obtained
by adding the edges of the distance-$2$ graph of $\G$
to the complete bipartite graph with bipartition $X_1 + X_2$,
and find that in each case,
there is a unique such $6$-coloring up to graph automorphism.
By identifying the partition of $\G$ into these color classes
with the partition of $X^{(2)}$ into $R_3$-cliques,
we see that the choice of such a graph for $\G^{(2)}_1$
uniquely determines the relation scheme $\A^{(2)}$.

For each of these graphs,
we thus build a candidate for the relation scheme $\A^{(2)}$
and attempt to compute the corresponding matrix $U$
with the coefficients of the unit vectors $u'_x \in S_1$ ($x \in X^{(2)}$)
using Algorithm~\ref{alg:unitvecs}.
We find that an embedding into $S_1$ exists
for $7$ choices of the graph $\G^{(2)}_1$.
Since $\tilde{2} = \{2\}$,
the corresponding relation schemes can be uniquely extended
into candidates for the relation scheme $\A^{(3)}$.
However,
we find that in none of these cases all the coefficients can be determined,
and thus conclude that none of these relation schemes
admit an embedding into $S_1$.
Since we have considered all the possibilities
for the induced subscheme $\A^{(3)}$,
it follows that the association scheme $\A$ does not exist.
\end{proof}

The
\href{https://nbviewer.org/github/jaanos/eigenspace-embeddings/blob/main/QPG5-6-45-22.ipynb}{\tt QPG5-6-45-22.ipynb}
notebook on the {\tt eigenspace-embeddings} repository~\cite{v25}
illustrates the computation needed to obtain the above result.

\section{Uniqueness results}
\label{sec:uniq}

Applying the same technique,
we obtain two more uniqueness results
for parameter sets of $5$-class association schemes.

\subsection{QPG with parameter array
$[[12, 2, 1, 12, 12], [6, 0, 4, 1; 0, 0, 1; 0, 1; 4]]$}
\label{ssec:p5-12-40-2}

Let $\A$ be a $5$-class association scheme with intersection numbers
\begin{equation}
\label{eqn:p5-12-40-2}
\begin{aligned}
(p^0_{ij})_{i,j=0}^5 &= \begin{pmatrix}
1 &  0 & 0 & 0 &  0 &  0 \\
0 & \mathit{12} & 0 & 0 &  0 &  0 \\
0 &  0 & \mathit{2} & 0 &  0 &  0 \\
0 &  0 & 0 & \mathit{1} &  0 &  0 \\
0 &  0 & 0 & 0 & \mathit{12} &  0 \\
0 &  0 & 0 & 0 &  0 & \mathit{12}
\end{pmatrix}, &
(p^1_{ij})_{i,j=0}^5 &= \begin{pmatrix}
0 & 1 & 0 & 0 & 0 & 0 \\
1 & 5 & 1 & 0 & 4 & 1 \\
0 & 1 & 0 & 0 & 0 & 1 \\
0 & 0 & 0 & 0 & 0 & 1 \\
0 & 4 & 0 & 0 & 4 & 4 \\
0 & 1 & 1 & 1 & 4 & 5
\end{pmatrix}, \\
(p^2_{ij})_{i,j=0}^5 &= \begin{pmatrix}
0 & 0 & 1 & 0 &  0 & 0 \\
0 & \mathit{6} & 0 & 0 &  0 & 6 \\
1 & 0 & 0 & 1 &  0 & 0 \\
0 & 0 & 1 & 0 &  0 & 0 \\
0 & 0 & 0 & 0 & 12 & 0 \\
0 & 0 & 0 & 0 &  0 & 6
\end{pmatrix}, &
(p^3_{ij})_{i,j=0}^5 &= \begin{pmatrix}
0 &  0 & 0 & 1 &  0 &  0 \\
0 &  \mathit{0} & \mathit{0} & 0 &  0 & 12 \\
0 &  \mathit{0} & 2 & 0 &  0 &  0 \\
1 &  0 & 0 & 0 &  0 &  0 \\
0 &  0 & 0 & 0 & 12 &  0 \\
0 & 12 & 0 & 0 &  0 &  0
\end{pmatrix}, \\
(p^4_{ij})_{i,j=0}^5 &= \begin{pmatrix}
0 & 0 & 0 & 0 & 1 & 0 \\
0 & \mathit{4} & \mathit{0} & \mathit{0} & 4 & 4 \\
0 & \mathit{0} & 0 & 0 & 2 & 0 \\
0 & \mathit{0} & 0 & 0 & 1 & 0 \\
1 & 4 & 2 & 1 & 0 & 4 \\
0 & 4 & 0 & 0 & 4 & 4
\end{pmatrix}, &
(p^5_{ij})_{i,j=0}^5 &= \begin{pmatrix}
0 & 0 & 0 & 0 & 0 & 1 \\
0 & \mathit{1} & \mathit{1} & \mathit{1} & \mathit{4} & 5 \\
0 & \mathit{1} & 0 & 0 & 0 & 1 \\
0 & \mathit{1} & 0 & 0 & 0 & 0 \\
0 & \mathit{4} & 0 & 0 & 4 & 4 \\
1 & 5 & 1 & 0 & 4 & 1
\end{pmatrix}.
\end{aligned}
\end{equation}
The graph $\G_1 = (X, R_1)$ is a quotient-polynomial graph on $40$ vertices
with parameter array $[[12, 2, 1, 12, 12],$ $[6, 0, 4, 1; 0, 0, 1; 0, 1; 4]]$.
The association scheme $\A$ is imprimitive
with imprimitivity set $\tilde{0} = \{0, 2, 3\}$.
The dual eigenmatrix of $\A$ is
$$
Q = \begin{pmatrix}
1 & 5 & 4 & 10 & 15 & 5 \\
1 & \frac{5}{2} & \frac{2}{3} & 0 & -\frac{5}{2} & -\frac{5}{3} \\
1 & 0 & 4 & -10 & 0 & 5 \\
1 & -5 & 4 & 10 & -15 & 5 \\
1 & 0 & -\frac{8}{3} & 0 & 0 & \frac{5}{3} \\
1 & -\frac{5}{2} & \frac{2}{3} & 0 & \frac{5}{2} & -\frac{5}{3}
\end{pmatrix}.
$$
By the ordering of eigenspaces used in the above matrix,
the corresponding dual imprimitivity set is $\overline{0} = \{0, 2, 5\}$.
We also have $\tilde{1} = \{1, 5\}$, $\tilde{4} = \{4\}$,
and $\overline{1} = \{1, 4\}$, $\overline{3} = \{3\}$.
Let $\{X_\ell \mid 1 \le \ell \le 10\}$
be the set of the equivalence classes of $R_{\tilde{0}}$,
and note that the graphs $\G_2|_{X_\ell} = (X_\ell, R_2|_{X_\ell})$
($1 \le \ell \le 10$)
are isomorphic to the complete bipartite graph $K_{2,2}$.
We will therefore call the sets $X_\ell$ ($1 \le \ell \le 10$)
the {\em $R_2$-bicliques} of $\A$.
We also note that the quotient scheme $\tilde{\A} = \A / \tilde{0}$
is the Johnson scheme $J(5, 2)$
-- i.e.,
the graph $\tilde{\G}_{\tilde{1}} = (\tilde{X}, \tilde{R}_{\tilde{1}})$
is isomorphic to the triangular graph $T(5)$,
and the graph $\tilde{\G}_{\tilde{4}} = (\tilde{X}, \tilde{R}_{\tilde{4}})$
is isomorphic to the Petersen graph.

We will consider embeddings
of subschemes induced on three $R_2$-bicliques of $\A$
into the eigenspace $S_1$ of dimension $m_1 = 5$.
We note that $\overline{m}_{\overline{1}} = 2$
and therefore ${m_1 \over \overline{m}_{\overline{1}}} = {5 \over 2} < 3$,
which may severely restrict
which of such subschemes admit an embedding into $S_1$.
We obtain the following result.

\begin{theorem}
\label{thm:p5-12-40-2}
There is, up to isomorphism,
precisely one association scheme $\A$
with intersection numbers \eqref{eqn:p5-12-40-2}.
\end{theorem}

\begin{proof}
Let $X_1, X_2, X_3$ be $R_2$-bicliques of $\A$
such that $(X_1, X_2) \in \tilde{R}_{\tilde{1}}$
and $(X_1, X_3), (X_2, X_3)$ $\in \tilde{R}_{\tilde{4}}$
(i.e., they correspond to a path of length $2$ in $\tilde{\G}_{\tilde{4}}$).
Let us first consider the graph $\G^{(2)}_1$.
Since $p^1_{12} = 1$ and $p^1_{13} = 0$,
and then $1 + p^1_{12} + p^1_{13} = p^5_{12} = p^5_{13} = 2$,
the graph $\G^{(2)}_1$ is a union of cycles of lengths divisible by $4$,
with the two neighbours of each vertex being in relation $R_2$.
In particular,
$\G^{(2)}_1$ must be isomorphic to $C_8$ or $2C_4$,
and its choice uniquely determines the relation scheme $\A^{(2)}$.
Since $\tilde{4} = \{4\}$,
such a relation scheme can be uniquely extended
into a candidate for the relation scheme $\A^{(3)}$.

We thus build two candidates for the relation scheme $\A^{(3)}$
and attempt to compute the corresponding matrix $U$
with the coefficients of the unit vectors $u'_x \in S_1$ ($x \in X^{(3)}$)
using Algorithm~\ref{alg:unitvecs}.
We find that an embedding into $S_1$ only exists
in the case when $\G^{(2)}_1 \cong 2C_4$.
The corresponding matrix $U$ has full column rank,
thus uniquely determining the orthonormal basis of $S_1$ being used.

By the above argument,
each vertex $y \in X \setminus X^{(3)}$ may be in relation $R_4$
with all vertices in at most one of $X_1, X_2, X_3$,
and is in relations $R_1$ and $R_5$ with a pair of vertices in relation $R_2$
from each of the remainder of these three $R_2$-bicliques.
Since there are $4$ such pairs in each of $X_\ell$ ($1 \le \ell \le 3$),
we examine the $3 \cdot 4^2 + 4^3 = 112$ candidates for such vertices $y$
and attempt to find the corresponding unit vectors $u'_y$.
Among them, we find $28$ vertices $y$ such that $u'_y$ is a unit vector,
which precisely matches the size of the set $X \setminus X^{(3)}$.
Let $Y$ be the set of all such vertices.
We define a graph $(Y, R')$,
where for a pair of vertices $y, z \in Y$
we have $(y, z) \in R'$ if and only if
$\langle u'_y, u'_z \rangle = {Q_{i1} \over m_1}$
for some $i$ ($1 \le i \le 5$).
We find that this graph is complete,
so we may attempt to construct $\A$
by identifying $Y$ with $X \setminus X^{(3)}$.

Note that since $Q_{21} = Q_{41} = 0$,
we cannot determine the relations of $\A$
from the inner products between vectors $u'_x$ ($x \in X$) alone.
We do however notice that the intersection numbers $p^i_{11}$ ($0 \le i \le 5$)
are all distinct,
so we may build a graph $\G = (X, R)$,
where $(x, y) \in R$ if and only if
$\langle u'_x, u'_y \rangle = {Q_{11} \over m_1} = {1 \over 2}$ ($x, y \in X$),
and then determine the relations $R_i$ ($0 \le i \le 5$)
so that $(x, y) \in R_i$ precisely when
$x$ and $y$ have exactly $p^i_{11}$ common neighbours in $\G$.
We verify that the obtained relation scheme
is an association scheme with the same parameters as $\A$,
and therefore conclude that there is,
up to isomorphism,
precisely one such association scheme.
\end{proof}

The
\href{https://nbviewer.org/github/jaanos/eigenspace-embeddings/blob/main/QPG5-12-40-2.ipynb}{\tt QPG5-12-40-2.ipynb}
notebook on the {\tt eigenspace-embeddings} repository~\cite{v25}
illustrates the computation needed to obtain the above result.

\begin{remark}
\label{rem:p5-12-40-2}
An alternative spherical representation of the association scheme $\A$
is given by the vectors
$u_{ij}^{\alpha \beta} = {\sqrt{2} \over 2} (\alpha e_i + \beta e_j)$
($1 \le i < j \le 5$, $\alpha, \beta \in \{-1, 1\}$),
where $\{e_i \mid 1 \le i \le 5\}$ is an orthonormal basis of $S_1$.
The inner product of two vectors
corresponding to a pair of vertices of $\A$ in relation $R_i$
equals ${Q_{i1} \over m_1}$ ($0 \le i \le 5$);
furthermore,
the vectors corresponding to a pair of vertices in relations $R_2$ and $R_4$
are $u_{ij}^{\alpha \beta}$ and $u_{i'j'}^{\alpha' \beta'}$
for some $i, i', j, j', \alpha, \alpha', \beta, \beta'$
with the sets $\{i, j\}$ and $\{i', j'\}$
being equal (with $\alpha \alpha' \beta \beta' = -1$) or disjoint,
respectively.

The graph $\G_1$ is a $6$-regular arc-transitive graph
of diameter $3$ and girth $3$.
Its group of automorphisms has order $3840$;
using GAP~\cite{g24},
we find that its structure can be described as $Z_2 \times (Z_2^4 \rtimes S_5)$
-- this also holds for the group of automorphisms of $\A$.
Its natural action on $X^2$ preserves the partition $\R$
-- i.e.,
each relation of $\A$ corresponds to an orbit of the action.
\end{remark}

\subsection{QPG with parameter array
$[[6, 4, 4, 12, 18], [3, 0, 0, 1; 0, 1, 0; 2, 0; 2]]$}
\label{ssec:p5-6-45-5}

Let $\A$ be a $5$-class association scheme with intersection numbers
\begin{equation}
\label{eqn:p5-6-45-5}
\begin{aligned}
(p^0_{ij})_{i,j=0}^5 &= \begin{pmatrix}
1 & 0 & 0 & 0 &  0 &  0 \\
0 & \mathit{6} & 0 & 0 &  0 &  0 \\
0 & 0 & \mathit{4} & 0 &  0 &  0 \\
0 & 0 & 0 & \mathit{4} &  0 &  0 \\
0 & 0 & 0 & 0 & \mathit{12} &  0 \\
0 & 0 & 0 & 0 &  0 & \mathit{18}
\end{pmatrix}, &
(p^1_{ij})_{i,j=0}^5 &= \begin{pmatrix}
0 & 1 & 0 & 0 & 0 & 0 \\
1 & 0 & 2 & 0 & 0 & 3 \\
0 & 2 & 0 & 0 & 2 & 0 \\
0 & 0 & 0 & 0 & 4 & 0 \\
0 & 0 & 2 & 4 & 0 & 6 \\
0 & 3 & 0 & 0 & 6 & 9
\end{pmatrix}, \\
(p^2_{ij})_{i,j=0}^5 &= \begin{pmatrix}
0 & 0 & 1 & 0 & 0 &  0 \\
0 & \mathit{3} & 0 & 0 & 3 &  0 \\
1 & 0 & 1 & 2 & 0 &  0 \\
0 & 0 & 2 & 2 & 0 &  0 \\
0 & 3 & 0 & 0 & 9 &  0 \\
0 & 0 & 0 & 0 & 0 & 18
\end{pmatrix}, &
(p^3_{ij})_{i,j=0}^5 &= \begin{pmatrix}
0 & 0 & 0 & 1 & 0 &  0 \\
0 & \mathit{0} & \mathit{0} & 0 & 6 &  0 \\
0 & \mathit{0} & 2 & 2 & 0 &  0 \\
1 & 0 & 2 & 1 & 0 &  0 \\
0 & 6 & 0 & 0 & 6 &  0 \\
0 & 0 & 0 & 0 & 0 & 18
\end{pmatrix}, \\
(p^4_{ij})_{i,j=0}^5 &= \begin{pmatrix}
0 & 0 & 0 & 0 & 1 & 0 \\
0 & \mathit{0} & \mathit{1} & \mathit{2} & 0 & 3 \\
0 & \mathit{1} & 0 & 0 & 3 & 0 \\
0 & \mathit{2} & 0 & 0 & 2 & 0 \\
1 & 0 & 3 & 2 & 0 & 6 \\
0 & 3 & 0 & 0 & 6 & 9
\end{pmatrix}, &
(p^5_{ij})_{i,j=0}^5 &= \begin{pmatrix}
0 & 0 & 0 & 0 & 0 & 1 \\
0 & \mathit{1} & \mathit{0} & \mathit{0} & \mathit{2} & 3 \\
0 & \mathit{0} & 0 & 0 & 0 & 4 \\
0 & \mathit{0} & 0 & 0 & 0 & 4 \\
0 & \mathit{2} & 0 & 0 & 4 & 6 \\
1 & 3 & 4 & 4 & 6 & 0
\end{pmatrix}.
\end{aligned}
\end{equation}
The graphs $\G_1 = (X, R_1)$ and $\G_4 = (X, R_4)$
are quotient-polynomial graphs on $45$ vertices
with parameter arrays $[[6, 4, 4, 12, 18], [3, 0, 0, 1; 0, 1, 0; 2, 0; 2]]$
and $[[12, 4, 4, 6, 18],$ $[9, 6, 0, 4; 0, 2, 0; 4, 0; 2]]$,
respectively.
The association scheme $\A$ is imprimitive
with imprimitivity set $\tilde{0} = \{0, 2, 3\}$.
The dual eigenmatrix of $\A$ is
$$
Q = \begin{pmatrix}
1 & 10 & 2 & 20 & 10 & 2 \\
1 & 5 & \frac{\sqrt{5} - 1}{2} & 0 & -5 & -\frac{\sqrt{5} + 1}{2} \\
1 & \frac{5}{2} & 2 & -10 & \frac{5}{2} & 2 \\
1 & -5 & 2 & 5 & -5 & 2 \\
1 & -\frac{5}{2} & \frac{\sqrt{5} - 1}{2} & 0 & \frac{5}{2} & -\frac{\sqrt{5} + 1}{2} \\
1 & 0 & -\frac{\sqrt{5} + 1}{2} & 0 & 0 & \frac{\sqrt{5} - 1}{2}
\end{pmatrix}.
$$
By the ordering of eigenspaces used in the above matrix,
the corresponding dual imprimitivity set is $\overline{0} = \{0, 2, 5\}$.
We also have $\tilde{1} = \{1, 4\}$, $\tilde{5} = \{5\}$,
and $\overline{1} = \{1, 4\}$, $\overline{3} = \{3\}$.
Let $\{X_\ell \mid 1 \le \ell \le 5\}$
be the set of the equivalence classes of $R_{\tilde{0}}$,
and note that the graphs $\G_2|_{X_\ell} = (X_\ell, R_2|_{X_\ell})$
($1 \le \ell \le 5$)
are isomorphic to the graph $K_3 \square K_3$,
or, equivalently, the Hamming graph $H(2, 3)$.
We have already encountered this situation in Subsection~\ref{ssec:p4-8-45-18},
and we will reuse the same terminology.
Furthermore,
we also note that the quotient scheme $\tilde{\A} = \A / \tilde{0}$
is the cyclic scheme $C_5$
-- i.e.,
the graphs $\tilde{\G}_{\tilde{1}} = (\tilde{X}, \tilde{R}_{\tilde{1}})$
and $\tilde{\G}_{\tilde{5}} = (\tilde{X}, \tilde{R}_{\tilde{5}})$
are both isomorphic to the graph $C_5$.

We will consider embeddings
of subschemes induced on three $(R_2 \cup R_3)$-cliques of $\A$
into the eigenspace $S_1$ of dimension $m_1 = 10$.
We note that $\overline{m}_{\overline{1}} = 4$
and therefore ${m_1 \over \overline{m}_{\overline{1}}} = {5 \over 2} < 3$,
which may severely restrict
which of such subschemes admit an embedding into $S_1$.
We obtain the following result.

\begin{theorem}
\label{thm:p5-6-45-5}
There is, up to isomorphism,
precisely one association scheme $\A$
with intersection numbers \eqref{eqn:p5-6-45-5}.
\end{theorem}

\begin{proof}
Let $X_1, X_2, X_3$ be $(R_2 \cup R_3)$-cliques of $\A$
such that $(X_1, X_2) \in \tilde{R}_{\tilde{1}}$
and $(X_1, X_3),$ $(X_2, X_3) \in \tilde{R}_{\tilde{5}}$
(i.e., they correspond to a path of length $2$ in $\tilde{\G}_{\tilde{5}}$).
Since $p^1_{12} = 2$ and $p^1_{13} = 0$,
we may apply Lemma~\ref{lem:hamming} to conclude that
the graph $\G^{(2)}_1$ is isomorphic to $3K_{3,3}$,
with the partitions of vertices of $X_1$ and $X_2$
corresponding to the connected components of $\G^{(2)}_1$
coincinding with spreads of $\G_2|_{X_1}$ and $\G_2|_{X_2}$.
The relation scheme $\A^{(2)}$ is thus uniquely determined,
and since $\tilde{5} = \{5\}$,
it can be uniquely extended into the relation scheme $\A^{(3)}$.

We attempt to compute the corresponding matrix $U$
with the coefficients of the unit vectors $u'_x \in S_1$ ($x \in X^{(3)}$)
using Algorithm~\ref{alg:unitvecs}
and obtain a matrix with full column rank,
thus uniquely determining the orthonormal basis of $S_1$ being used.

By the above argument,
each vertex $y \in X \setminus X^{(3)}$ is in relation $R_5$
with all vertices in one of $X_1$ and $X_2$,
in relation $R_1$ with six vertices forming two lines
in each of the remaining $(R_2 \cup R_3)$-cliques within $X^{(3)}$,
and in relation $R_4$ with the other vertices of $X^{(3)}$.
Since there are $6$ lines in each of $X_\ell$ ($1 \le \ell \le 3$),
we examine the $2 \cdot 6^2 = 72$ candidates for such vertices $y$
and attempt to find the corresponding unit vectors $u'_y$.
Among them, we find $36$ vertices $y$ such that $u'_y$ is a unit vector,
which is precisely twice the size of the set $X \setminus X^{(3)}$.

Let $Y$ be the set of all such vertices.
Since $X \setminus X^{(3)} = X_4 \cup X_5$
and $(X_4, X_5) \in \tilde{R}_{\tilde{5}}$,
we define a graph $(Y, R')$,
where for a pair of vertices $y, z \in Y$
we have $(y, z) \in R'$ if and only if
$\langle u'_y, u'_z \rangle = {Q_{i1} \over m_1}$ for some $i \in \{2, 3, 5\}$.
We find that this graph consists of two disjoint $18$-cliques
-- we label their vertex sets by $Y_1$ and $Y_2$.
We may then attempt to construct $\A$
by identifying $Y_1$ or $Y_2$ with $X \setminus X^{(3)}$
and using the inner products among the vectors $u'_x$ ($x \in X$)
to determine their relations.
We verify that both the obtained relation schemes
are association schemes with the same parameters as $\A$.
Since the two association schemes are isomorphic,
we conclude that there is, up to isomorphism,
precisely one association scheme with the parameters of $\A$.
\end{proof}

The
\href{https://nbviewer.org/github/jaanos/eigenspace-embeddings/blob/main/QPG5-6-45-5.ipynb}{\tt QPG5-6-45-5.ipynb}
notebook on the {\tt eigenspace-embeddings} repository~\cite{v25}
illustrates the computation needed to obtain the above result.

\begin{remark}
\label{rem:p5-6-45-5}
Uniqueness can be also proved without computing a spherical representation:
Lemma~\ref{lem:hamming} and the intersection number $p^5_{11} = 1$
imply that for each $(R_2 \cup R_3)$-clique $Y$,
different spreads of $\G_2|_Y$ are used to determine
which vertices of the two $(R_2 \cup R_3)$-cliques $Y'$ and $Y''$
such that $(Y, Y'), (Y, Y'') \in \tilde{R}_{\tilde{1}}$
are in relation $R_1$ with the vertices of $Y$.
This uniquely determines the structure of $\A$,
as shown in Figure~\ref{fig:c352}.

The graph $\G_1 = (X, R_1)$ is a $6$-regular arc-transitive graph
of diameter $4$ and girth $4$.
It can be described by taking the vertex set $X = \Z_5 \times \Z_3 \times \Z_3$
and the relation $R_1 = \vec{R}_1 \cup \vec{R}_1^\top$,
where
$\vec{R}_1 = \{((\ell, r, s), (\ell+1, s, t)) \mid \ell \in \Z_5, r, s, t \in \Z_3\}$.
From this description we can see
that $\G_1$ is isomorphic to the graph $C(3, 5, 2)$
as defined by Praeger and Xu~\cite{px89}.
Its group of automorphisms has (large) order $77760$
-- this also holds for the group of automorphisms of $\A$.
Its natural action on $X^2$ preserves the partition $\R$
-- i.e.,
each relation of $\A$ corresponds to an orbit of the action.
Note in particular that,
although the graphs $\G_2|_{X_\ell}$ and $\G_3|_{X_\ell}$ ($1 \le \ell \le 5$)
are mutually isomorphic,
no automorphism of $\A$ exchanges $R_2$ and $R_3$
(as it can be seen from, e.g., $p^1_{12} \ne p^1_{13}$).
\end{remark}

\begin{figure}[t]
\makebox[\textwidth][c]{
	\beginpgfgraphicnamed{fig-c352}
	\begin{tikzpicture}[style=thick,scale=0.7]
		\input{c352.tikz}
	\end{tikzpicture}
\endpgfgraphicnamed
}

\caption{The association scheme $\A$.
The lines are represented with rounded rectangles
and form five $\GQ(2, 1)$ geometries;
the vertices are implied at intersections of lines.
Two distinct vertices are in relation $R_1$
if they are contained in two lines connected by an edge,
in relation $R_2$ if they are contained in a common line,
in relation $R_3$ if they are contained in distinct lines
of the same $\GQ(2, 1)$,
in relation $R_4$ if they are contained
in lines of adjacent $\GQ(2, 1)$ geometries
not connected by an edge,
and in relation $R_5$ otherwise.
}
\label{fig:c352}
\end{figure}
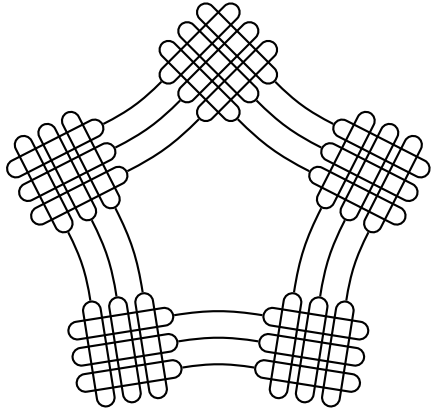

\subsection*{Acknowledgements}

This work is supported in part by the Slovenian Research and Innovation Agency
(research program P1-0285 and research projects J1-3002, N1-0216 and N1-0218).

\clearpage

\input{embeddings-appendix.tex}



\bibliographystyle{abbrv}

%% file: a3.tikz
\tikzstyle{every node}=[draw, rounded rectangle, inner xsep=0.7cm]

\node (ay11) at ($(-6, 3) + (xyz polar cs:angle= 90,radius=1.5) + (xyz polar cs:angle= 90,radius= 0.5)$) [rotate=  0] {};
\node (ay12) at ($(-6, 3) + (xyz polar cs:angle= 90,radius=1.5) + (xyz polar cs:angle= 90,radius= 0  )$) [rotate=  0] {};
\node (ay13) at ($(-6, 3) + (xyz polar cs:angle= 90,radius=1.5) + (xyz polar cs:angle= 90,radius=-0.5)$) [rotate=  0] {};
\node (az11) at ($(-6, 3) + (xyz polar cs:angle= 90,radius=1.5) + (xyz polar cs:angle=  0,radius= 0.5)$) [rotate= 90] {};
\node (az12) at ($(-6, 3) + (xyz polar cs:angle= 90,radius=1.5) + (xyz polar cs:angle=  0,radius= 0  )$) [rotate= 90] {};
\node (az13) at ($(-6, 3) + (xyz polar cs:angle= 90,radius=1.5) + (xyz polar cs:angle=  0,radius=-0.5)$) [rotate= 90] {};
\node (ay21) at ($(-6, 3) + (xyz polar cs:angle=210,radius=1.5) + (xyz polar cs:angle=210,radius= 0.5)$) [rotate=120] {};
\node (ay22) at ($(-6, 3) + (xyz polar cs:angle=210,radius=1.5) + (xyz polar cs:angle=210,radius= 0  )$) [rotate=120] {};
\node (ay23) at ($(-6, 3) + (xyz polar cs:angle=210,radius=1.5) + (xyz polar cs:angle=210,radius=-0.5)$) [rotate=120] {};
\node (az21) at ($(-6, 3) + (xyz polar cs:angle=210,radius=1.5) + (xyz polar cs:angle=120,radius= 0.5)$) [rotate=210] {};
\node (az22) at ($(-6, 3) + (xyz polar cs:angle=210,radius=1.5) + (xyz polar cs:angle=120,radius= 0  )$) [rotate=210] {};
\node (az23) at ($(-6, 3) + (xyz polar cs:angle=210,radius=1.5) + (xyz polar cs:angle=120,radius=-0.5)$) [rotate=210] {};
\node (ay31) at ($(-6, 3) + (xyz polar cs:angle=330,radius=1.5) + (xyz polar cs:angle=330,radius= 0.5)$) [rotate=240] {};
\node (ay32) at ($(-6, 3) + (xyz polar cs:angle=330,radius=1.5) + (xyz polar cs:angle=330,radius= 0  )$) [rotate=240] {};
\node (ay33) at ($(-6, 3) + (xyz polar cs:angle=330,radius=1.5) + (xyz polar cs:angle=330,radius=-0.5)$) [rotate=240] {};
\node (az31) at ($(-6, 3) + (xyz polar cs:angle=330,radius=1.5) + (xyz polar cs:angle=240,radius= 0.5)$) [rotate=330] {};
\node (az32) at ($(-6, 3) + (xyz polar cs:angle=330,radius=1.5) + (xyz polar cs:angle=240,radius= 0  )$) [rotate=330] {};
\node (az33) at ($(-6, 3) + (xyz polar cs:angle=330,radius=1.5) + (xyz polar cs:angle=240,radius=-0.5)$) [rotate=330] {};

\draw (ay11.west) to[bend right=60] (ay21.east);
\draw (ay12.west) to[bend right=60] (ay22.east);
\draw (ay13.west) to[bend right=60] (ay23.east);
\draw (ay21.west) to[bend right=60] (ay31.east);
\draw (ay22.west) to[bend right=60] (ay32.east);
\draw (ay23.west) to[bend right=60] (ay33.east);
\draw (ay31.west) to[bend right=60] (ay11.east);
\draw (ay32.west) to[bend right=60] (ay12.east);
\draw (ay33.west) to[bend right=60] (ay13.east);

\node (by11) at ($( 0, 3) + (xyz polar cs:angle= 90,radius=1.5) + (xyz polar cs:angle= 90,radius= 0.5)$) [rotate=  0] {};
\node (by12) at ($( 0, 3) + (xyz polar cs:angle= 90,radius=1.5) + (xyz polar cs:angle= 90,radius= 0  )$) [rotate=  0] {};
\node (by13) at ($( 0, 3) + (xyz polar cs:angle= 90,radius=1.5) + (xyz polar cs:angle= 90,radius=-0.5)$) [rotate=  0] {};
\node (bz11) at ($( 0, 3) + (xyz polar cs:angle= 90,radius=1.5) + (xyz polar cs:angle=  0,radius= 0.5)$) [rotate= 90] {};
\node (bz12) at ($( 0, 3) + (xyz polar cs:angle= 90,radius=1.5) + (xyz polar cs:angle=  0,radius= 0  )$) [rotate= 90] {};
\node (bz13) at ($( 0, 3) + (xyz polar cs:angle= 90,radius=1.5) + (xyz polar cs:angle=  0,radius=-0.5)$) [rotate= 90] {};
\node (by21) at ($( 0, 3) + (xyz polar cs:angle=210,radius=1.5) + (xyz polar cs:angle=210,radius= 0.5)$) [rotate=120] {};
\node (by22) at ($( 0, 3) + (xyz polar cs:angle=210,radius=1.5) + (xyz polar cs:angle=210,radius= 0  )$) [rotate=120] {};
\node (by23) at ($( 0, 3) + (xyz polar cs:angle=210,radius=1.5) + (xyz polar cs:angle=210,radius=-0.5)$) [rotate=120] {};
\node (bz21) at ($( 0, 3) + (xyz polar cs:angle=210,radius=1.5) + (xyz polar cs:angle=120,radius= 0.5)$) [rotate=210] {};
\node (bz22) at ($( 0, 3) + (xyz polar cs:angle=210,radius=1.5) + (xyz polar cs:angle=120,radius= 0  )$) [rotate=210] {};
\node (bz23) at ($( 0, 3) + (xyz polar cs:angle=210,radius=1.5) + (xyz polar cs:angle=120,radius=-0.5)$) [rotate=210] {};
\node (by31) at ($( 0, 3) + (xyz polar cs:angle=330,radius=1.5) + (xyz polar cs:angle=330,radius= 0.5)$) [rotate=240] {};
\node (by32) at ($( 0, 3) + (xyz polar cs:angle=330,radius=1.5) + (xyz polar cs:angle=330,radius= 0  )$) [rotate=240] {};
\node (by33) at ($( 0, 3) + (xyz polar cs:angle=330,radius=1.5) + (xyz polar cs:angle=330,radius=-0.5)$) [rotate=240] {};
\node (bz31) at ($( 0, 3) + (xyz polar cs:angle=330,radius=1.5) + (xyz polar cs:angle=240,radius= 0.5)$) [rotate=330] {};
\node (bz32) at ($( 0, 3) + (xyz polar cs:angle=330,radius=1.5) + (xyz polar cs:angle=240,radius= 0  )$) [rotate=330] {};
\node (bz33) at ($( 0, 3) + (xyz polar cs:angle=330,radius=1.5) + (xyz polar cs:angle=240,radius=-0.5)$) [rotate=330] {};

\draw (by11.west) to[bend right=60] (by21.east);
\draw (by12.west) to[bend right=60] (by22.east);
\draw (by13.west) to[bend right=60] (by23.east);
\draw (by21.west) to[bend right=50] (by32.east);
\draw (by22.west) to[bend right=50] (by31.east);
\draw (by23.west) to[bend right=60] (by33.east);
\draw (by31.west) to[bend right=60] (by11.east);
\draw (by32.west) to[bend right=60] (by12.east);
\draw (by33.west) to[bend right=60] (by13.east);

\node (cy11) at ($( 6, 3) + (xyz polar cs:angle= 90,radius=1.5) + (xyz polar cs:angle= 90,radius= 0.5)$) [rotate=  0] {};
\node (cy12) at ($( 6, 3) + (xyz polar cs:angle= 90,radius=1.5) + (xyz polar cs:angle= 90,radius= 0  )$) [rotate=  0] {};
\node (cy13) at ($( 6, 3) + (xyz polar cs:angle= 90,radius=1.5) + (xyz polar cs:angle= 90,radius=-0.5)$) [rotate=  0] {};
\node (cz11) at ($( 6, 3) + (xyz polar cs:angle= 90,radius=1.5) + (xyz polar cs:angle=  0,radius= 0.5)$) [rotate= 90] {};
\node (cz12) at ($( 6, 3) + (xyz polar cs:angle= 90,radius=1.5) + (xyz polar cs:angle=  0,radius= 0  )$) [rotate= 90] {};
\node (cz13) at ($( 6, 3) + (xyz polar cs:angle= 90,radius=1.5) + (xyz polar cs:angle=  0,radius=-0.5)$) [rotate= 90] {};
\node (cy21) at ($( 6, 3) + (xyz polar cs:angle=210,radius=1.5) + (xyz polar cs:angle=210,radius= 0.5)$) [rotate=120] {};
\node (cy22) at ($( 6, 3) + (xyz polar cs:angle=210,radius=1.5) + (xyz polar cs:angle=210,radius= 0  )$) [rotate=120] {};
\node (cy23) at ($( 6, 3) + (xyz polar cs:angle=210,radius=1.5) + (xyz polar cs:angle=210,radius=-0.5)$) [rotate=120] {};
\node (cz21) at ($( 6, 3) + (xyz polar cs:angle=210,radius=1.5) + (xyz polar cs:angle=120,radius= 0.5)$) [rotate=210] {};
\node (cz22) at ($( 6, 3) + (xyz polar cs:angle=210,radius=1.5) + (xyz polar cs:angle=120,radius= 0  )$) [rotate=210] {};
\node (cz23) at ($( 6, 3) + (xyz polar cs:angle=210,radius=1.5) + (xyz polar cs:angle=120,radius=-0.5)$) [rotate=210] {};
\node (cy31) at ($( 6, 3) + (xyz polar cs:angle=330,radius=1.5) + (xyz polar cs:angle=330,radius= 0.5)$) [rotate=240] {};
\node (cy32) at ($( 6, 3) + (xyz polar cs:angle=330,radius=1.5) + (xyz polar cs:angle=330,radius= 0  )$) [rotate=240] {};
\node (cy33) at ($( 6, 3) + (xyz polar cs:angle=330,radius=1.5) + (xyz polar cs:angle=330,radius=-0.5)$) [rotate=240] {};
\node (cz31) at ($( 6, 3) + (xyz polar cs:angle=330,radius=1.5) + (xyz polar cs:angle=240,radius= 0.5)$) [rotate=330] {};
\node (cz32) at ($( 6, 3) + (xyz polar cs:angle=330,radius=1.5) + (xyz polar cs:angle=240,radius= 0  )$) [rotate=330] {};
\node (cz33) at ($( 6, 3) + (xyz polar cs:angle=330,radius=1.5) + (xyz polar cs:angle=240,radius=-0.5)$) [rotate=330] {};

\draw (cy11.west) to[bend right=60] (cy21.east);
\draw (cy12.west) to[bend right=60] (cy22.east);
\draw (cy13.west) to[bend right=60] (cy23.east);
\draw (cy21.west) to[bend right=50] (cy32.east);
\draw (cy22.west) to[bend right=50] (cy33.east);
\draw (cy23.west) to[bend right=50] (cy31.east);
\draw (cy31.west) to[bend right=60] (cy11.east);
\draw (cy32.west) to[bend right=60] (cy12.east);
\draw (cy33.west) to[bend right=60] (cy13.east);

\node (dy11) at ($(-6,-3) + (xyz polar cs:angle= 90,radius=1.5) + (xyz polar cs:angle=135,radius= 0.5)$) [rotate= 45] {};
\node (dy12) at ($(-6,-3) + (xyz polar cs:angle= 90,radius=1.5) + (xyz polar cs:angle=135,radius= 0  )$) [rotate= 45] {};
\node (dy13) at ($(-6,-3) + (xyz polar cs:angle= 90,radius=1.5) + (xyz polar cs:angle=135,radius=-0.5)$) [rotate= 45] {};
\node (dz11) at ($(-6,-3) + (xyz polar cs:angle= 90,radius=1.5) + (xyz polar cs:angle= 45,radius= 0.5)$) [rotate=135] {};
\node (dz12) at ($(-6,-3) + (xyz polar cs:angle= 90,radius=1.5) + (xyz polar cs:angle= 45,radius= 0  )$) [rotate=135] {};
\node (dz13) at ($(-6,-3) + (xyz polar cs:angle= 90,radius=1.5) + (xyz polar cs:angle= 45,radius=-0.5)$) [rotate=135] {};
\node (dy21) at ($(-6,-3) + (xyz polar cs:angle=210,radius=1.5) + (xyz polar cs:angle=180,radius= 0.5)$) [rotate= 90] {};
\node (dy22) at ($(-6,-3) + (xyz polar cs:angle=210,radius=1.5) + (xyz polar cs:angle=180,radius= 0  )$) [rotate= 90] {};
\node (dy23) at ($(-6,-3) + (xyz polar cs:angle=210,radius=1.5) + (xyz polar cs:angle=180,radius=-0.5)$) [rotate= 90] {};
\node (dz21) at ($(-6,-3) + (xyz polar cs:angle=210,radius=1.5) + (xyz polar cs:angle= 90,radius= 0.5)$) [rotate=180] {};
\node (dz22) at ($(-6,-3) + (xyz polar cs:angle=210,radius=1.5) + (xyz polar cs:angle= 90,radius= 0  )$) [rotate=180] {};
\node (dz23) at ($(-6,-3) + (xyz polar cs:angle=210,radius=1.5) + (xyz polar cs:angle= 90,radius=-0.5)$) [rotate=180] {};
\node (dy31) at ($(-6,-3) + (xyz polar cs:angle=330,radius=1.5) + (xyz polar cs:angle=  0,radius= 0.5)$) [rotate=270] {};
\node (dy32) at ($(-6,-3) + (xyz polar cs:angle=330,radius=1.5) + (xyz polar cs:angle=  0,radius= 0  )$) [rotate=270] {};
\node (dy33) at ($(-6,-3) + (xyz polar cs:angle=330,radius=1.5) + (xyz polar cs:angle=  0,radius=-0.5)$) [rotate=270] {};
\node (dz31) at ($(-6,-3) + (xyz polar cs:angle=330,radius=1.5) + (xyz polar cs:angle=270,radius= 0.5)$) [rotate=  0] {};
\node (dz32) at ($(-6,-3) + (xyz polar cs:angle=330,radius=1.5) + (xyz polar cs:angle=270,radius= 0  )$) [rotate=  0] {};
\node (dz33) at ($(-6,-3) + (xyz polar cs:angle=330,radius=1.5) + (xyz polar cs:angle=270,radius=-0.5)$) [rotate=  0] {};

\draw (dy11.west) to[bend right=15] (dy21.east);
\draw (dy12.west) to[bend right=15] (dy22.east);
\draw (dy13.west) to[bend right=15] (dy23.east);
\draw (dy21.west) to[bend right=90] (dy31.east);
\draw (dy22.west) to[bend right=90] (dy32.east);
\draw (dy23.west) to[bend right=90] (dy33.east);
\draw (dy31.west) to[bend right=15] (dz11.west);
\draw (dy32.west) to[bend right=15] (dz12.west);
\draw (dy33.west) to[bend right=15] (dz13.west);

\node (ey11) at ($( 0,-3) + (xyz polar cs:angle= 90,radius=1.5) + (xyz polar cs:angle= 90,radius= 0.5)$) [rotate=  0] {};
\node (ey12) at ($( 0,-3) + (xyz polar cs:angle= 90,radius=1.5) + (xyz polar cs:angle= 90,radius= 0  )$) [rotate=  0] {};
\node (ey13) at ($( 0,-3) + (xyz polar cs:angle= 90,radius=1.5) + (xyz polar cs:angle= 90,radius=-0.5)$) [rotate=  0] {};
\node (ez11) at ($( 0,-3) + (xyz polar cs:angle= 90,radius=1.5) + (xyz polar cs:angle=  0,radius= 0.5)$) [rotate= 90] {};
\node (ez12) at ($( 0,-3) + (xyz polar cs:angle= 90,radius=1.5) + (xyz polar cs:angle=  0,radius= 0  )$) [rotate= 90] {};
\node (ez13) at ($( 0,-3) + (xyz polar cs:angle= 90,radius=1.5) + (xyz polar cs:angle=  0,radius=-0.5)$) [rotate= 90] {};
\node (ey21) at ($( 0,-3) + (xyz polar cs:angle=210,radius=1.5) + (xyz polar cs:angle=210,radius= 0.5)$) [rotate=120] {};
\node (ey22) at ($( 0,-3) + (xyz polar cs:angle=210,radius=1.5) + (xyz polar cs:angle=210,radius= 0  )$) [rotate=120] {};
\node (ey23) at ($( 0,-3) + (xyz polar cs:angle=210,radius=1.5) + (xyz polar cs:angle=210,radius=-0.5)$) [rotate=120] {};
\node (ez21) at ($( 0,-3) + (xyz polar cs:angle=210,radius=1.5) + (xyz polar cs:angle=120,radius= 0.5)$) [rotate=210] {};
\node (ez22) at ($( 0,-3) + (xyz polar cs:angle=210,radius=1.5) + (xyz polar cs:angle=120,radius= 0  )$) [rotate=210] {};
\node (ez23) at ($( 0,-3) + (xyz polar cs:angle=210,radius=1.5) + (xyz polar cs:angle=120,radius=-0.5)$) [rotate=210] {};
\node (ey31) at ($( 0,-3) + (xyz polar cs:angle=330,radius=1.5) + (xyz polar cs:angle=330,radius= 0.5)$) [rotate=240] {};
\node (ey32) at ($( 0,-3) + (xyz polar cs:angle=330,radius=1.5) + (xyz polar cs:angle=330,radius= 0  )$) [rotate=240] {};
\node (ey33) at ($( 0,-3) + (xyz polar cs:angle=330,radius=1.5) + (xyz polar cs:angle=330,radius=-0.5)$) [rotate=240] {};
\node (ez31) at ($( 0,-3) + (xyz polar cs:angle=330,radius=1.5) + (xyz polar cs:angle=240,radius= 0.5)$) [rotate=330] {};
\node (ez32) at ($( 0,-3) + (xyz polar cs:angle=330,radius=1.5) + (xyz polar cs:angle=240,radius= 0  )$) [rotate=330] {};
\node (ez33) at ($( 0,-3) + (xyz polar cs:angle=330,radius=1.5) + (xyz polar cs:angle=240,radius=-0.5)$) [rotate=330] {};

\draw (ey11.west) to[bend right=60] (ey21.east);
\draw (ey12.west) to[bend right=60] (ey22.east);
\draw (ey13.west) to[bend right=60] (ey23.east);
\draw (ez21.west) to[bend  left=30] (ez33.west);
\draw (ez22.west) to[bend  left=30] (ez32.west);
\draw (ez23.west) to[bend  left=30] (ez31.west);
\draw (ey31.west) to[bend right=60] (ey11.east);
\draw (ey32.west) to[bend right=60] (ey12.east);
\draw (ey33.west) to[bend right=60] (ey13.east);

\node (fy11) at ($( 6,-3) + (xyz polar cs:angle= 90,radius=1.5) + (xyz polar cs:angle=135,radius= 0.5)$) [rotate= 45] {};
\node (fy12) at ($( 6,-3) + (xyz polar cs:angle= 90,radius=1.5) + (xyz polar cs:angle=135,radius= 0  )$) [rotate= 45] {};
\node (fy13) at ($( 6,-3) + (xyz polar cs:angle= 90,radius=1.5) + (xyz polar cs:angle=135,radius=-0.5)$) [rotate= 45] {};
\node (fz11) at ($( 6,-3) + (xyz polar cs:angle= 90,radius=1.5) + (xyz polar cs:angle= 45,radius= 0.5)$) [rotate=135] {};
\node (fz12) at ($( 6,-3) + (xyz polar cs:angle= 90,radius=1.5) + (xyz polar cs:angle= 45,radius= 0  )$) [rotate=135] {};
\node (fz13) at ($( 6,-3) + (xyz polar cs:angle= 90,radius=1.5) + (xyz polar cs:angle= 45,radius=-0.5)$) [rotate=135] {};
\node (fy21) at ($( 6,-3) + (xyz polar cs:angle=210,radius=1.5) + (xyz polar cs:angle=255,radius= 0.5)$) [rotate=165] {};
\node (fy22) at ($( 6,-3) + (xyz polar cs:angle=210,radius=1.5) + (xyz polar cs:angle=255,radius= 0  )$) [rotate=165] {};
\node (fy23) at ($( 6,-3) + (xyz polar cs:angle=210,radius=1.5) + (xyz polar cs:angle=255,radius=-0.5)$) [rotate=165] {};
\node (fz21) at ($( 6,-3) + (xyz polar cs:angle=210,radius=1.5) + (xyz polar cs:angle=165,radius= 0.5)$) [rotate=255] {};
\node (fz22) at ($( 6,-3) + (xyz polar cs:angle=210,radius=1.5) + (xyz polar cs:angle=165,radius= 0  )$) [rotate=255] {};
\node (fz23) at ($( 6,-3) + (xyz polar cs:angle=210,radius=1.5) + (xyz polar cs:angle=165,radius=-0.5)$) [rotate=255] {};
\node (fy31) at ($( 6,-3) + (xyz polar cs:angle=330,radius=1.5) + (xyz polar cs:angle= 15,radius= 0.5)$) [rotate=285] {};
\node (fy32) at ($( 6,-3) + (xyz polar cs:angle=330,radius=1.5) + (xyz polar cs:angle= 15,radius= 0  )$) [rotate=285] {};
\node (fy33) at ($( 6,-3) + (xyz polar cs:angle=330,radius=1.5) + (xyz polar cs:angle= 15,radius=-0.5)$) [rotate=285] {};
\node (fz31) at ($( 6,-3) + (xyz polar cs:angle=330,radius=1.5) + (xyz polar cs:angle=285,radius= 0.5)$) [rotate= 15] {};
\node (fz32) at ($( 6,-3) + (xyz polar cs:angle=330,radius=1.5) + (xyz polar cs:angle=285,radius= 0  )$) [rotate= 15] {};
\node (fz33) at ($( 6,-3) + (xyz polar cs:angle=330,radius=1.5) + (xyz polar cs:angle=285,radius=-0.5)$) [rotate= 15] {};

\draw (fy11.west) to[bend right=15] (fz21.west);
\draw (fy12.west) to[bend right=15] (fz22.west);
\draw (fy13.west) to[bend right=15] (fz23.west);
\draw (fy21.west) to[bend right=15] (fz31.west);
\draw (fy22.west) to[bend right=15] (fz32.west);
\draw (fy23.west) to[bend right=15] (fz33.west);
\draw (fy31.west) to[bend right=15] (fz11.west);
\draw (fy32.west) to[bend right=15] (fz12.west);
\draw (fy33.west) to[bend right=15] (fz13.west);

%% file: c352.tikz
\tikzstyle{every node}=[draw, rounded rectangle, inner xsep=0.7cm]

\node (y11) at ($(xyz polar cs:angle= 90,radius=3) + (xyz polar cs:angle=135,radius= 0.5)$) [rotate= 45] {};
\node (y12) at ($(xyz polar cs:angle= 90,radius=3) + (xyz polar cs:angle=135,radius= 0  )$) [rotate= 45] {};
\node (y13) at ($(xyz polar cs:angle= 90,radius=3) + (xyz polar cs:angle=135,radius=-0.5)$) [rotate= 45] {};
\node (z11) at ($(xyz polar cs:angle= 90,radius=3) + (xyz polar cs:angle= 45,radius= 0.5)$) [rotate=135] {};
\node (z12) at ($(xyz polar cs:angle= 90,radius=3) + (xyz polar cs:angle= 45,radius= 0  )$) [rotate=135] {};
\node (z13) at ($(xyz polar cs:angle= 90,radius=3) + (xyz polar cs:angle= 45,radius=-0.5)$) [rotate=135] {};
\node (y21) at ($(xyz polar cs:angle=162,radius=3) + (xyz polar cs:angle=207,radius= 0.5)$) [rotate=117] {};
\node (y22) at ($(xyz polar cs:angle=162,radius=3) + (xyz polar cs:angle=207,radius= 0  )$) [rotate=117] {};
\node (y23) at ($(xyz polar cs:angle=162,radius=3) + (xyz polar cs:angle=207,radius=-0.5)$) [rotate=117] {};
\node (z21) at ($(xyz polar cs:angle=162,radius=3) + (xyz polar cs:angle=117,radius= 0.5)$) [rotate=207] {};
\node (z22) at ($(xyz polar cs:angle=162,radius=3) + (xyz polar cs:angle=117,radius= 0  )$) [rotate=207] {};
\node (z23) at ($(xyz polar cs:angle=162,radius=3) + (xyz polar cs:angle=117,radius=-0.5)$) [rotate=207] {};
\node (y31) at ($(xyz polar cs:angle=234,radius=3) + (xyz polar cs:angle=279,radius= 0.5)$) [rotate=189] {};
\node (y32) at ($(xyz polar cs:angle=234,radius=3) + (xyz polar cs:angle=279,radius= 0  )$) [rotate=189] {};
\node (y33) at ($(xyz polar cs:angle=234,radius=3) + (xyz polar cs:angle=279,radius=-0.5)$) [rotate=189] {};
\node (z31) at ($(xyz polar cs:angle=234,radius=3) + (xyz polar cs:angle=189,radius= 0.5)$) [rotate=279] {};
\node (z32) at ($(xyz polar cs:angle=234,radius=3) + (xyz polar cs:angle=189,radius= 0  )$) [rotate=279] {};
\node (z33) at ($(xyz polar cs:angle=234,radius=3) + (xyz polar cs:angle=189,radius=-0.5)$) [rotate=279] {};
\node (y41) at ($(xyz polar cs:angle=306,radius=3) + (xyz polar cs:angle=351,radius= 0.5)$) [rotate=261] {};
\node (y42) at ($(xyz polar cs:angle=306,radius=3) + (xyz polar cs:angle=351,radius= 0  )$) [rotate=261] {};
\node (y43) at ($(xyz polar cs:angle=306,radius=3) + (xyz polar cs:angle=351,radius=-0.5)$) [rotate=261] {};
\node (z41) at ($(xyz polar cs:angle=306,radius=3) + (xyz polar cs:angle=261,radius= 0.5)$) [rotate=351] {};
\node (z42) at ($(xyz polar cs:angle=306,radius=3) + (xyz polar cs:angle=261,radius= 0  )$) [rotate=351] {};
\node (z43) at ($(xyz polar cs:angle=306,radius=3) + (xyz polar cs:angle=261,radius=-0.5)$) [rotate=351] {};
\node (y51) at ($(xyz polar cs:angle= 18,radius=3) + (xyz polar cs:angle= 63,radius= 0.5)$) [rotate=333] {};
\node (y52) at ($(xyz polar cs:angle= 18,radius=3) + (xyz polar cs:angle= 63,radius= 0  )$) [rotate=333] {};
\node (y53) at ($(xyz polar cs:angle= 18,radius=3) + (xyz polar cs:angle= 63,radius=-0.5)$) [rotate=333] {};
\node (z51) at ($(xyz polar cs:angle= 18,radius=3) + (xyz polar cs:angle=333,radius= 0.5)$) [rotate= 63] {};
\node (z52) at ($(xyz polar cs:angle= 18,radius=3) + (xyz polar cs:angle=333,radius= 0  )$) [rotate= 63] {};
\node (z53) at ($(xyz polar cs:angle= 18,radius=3) + (xyz polar cs:angle=333,radius=-0.5)$) [rotate= 63] {};

\draw (y11.west) to[bend left=9] (z21.west);
\draw (y12.west) to[bend left=9] (z22.west);
\draw (y13.west) to[bend left=9] (z23.west);
\draw (y21.west) to[bend left=9] (z31.west);
\draw (y22.west) to[bend left=9] (z32.west);
\draw (y23.west) to[bend left=9] (z33.west);
\draw (y31.west) to[bend left=9] (z41.west);
\draw (y32.west) to[bend left=9] (z42.west);
\draw (y33.west) to[bend left=9] (z43.west);
\draw (y41.west) to[bend left=9] (z51.west);
\draw (y42.west) to[bend left=9] (z52.west);
\draw (y43.west) to[bend left=9] (z53.west);
\draw (y51.west) to[bend left=9] (z11.west);
\draw (y52.west) to[bend left=9] (z12.west);
\draw (y53.west) to[bend left=9] (z13.west);

%% file: embeddings-appendix.tex
\appendix
\section{Known constructions and nonexistences}
\label{app:known}

Here,
we present the results of feasibility checking
on parameter arrays for relational quotient-polynomial graphs
marked as feasible in~\cite{hm23a}.
The following tables present the parameter arrays
together with reasons for nonexistence or constructions.
In the cases when multiple parameter arrays
correspond to the same parameter set of an association scheme
(i.e., given with different orderings of the relations),
they are listed together,
and the ordering of the relations
(using the index set $\I = \{0, 1, \dots, d\}$
for $d$-class association schemes,
where $R_0$ is the identity relation
and $R_1$ is the adjacency relation of the graph)
in the given reason for nonexistence
corresponds to the ordering in the first listed parameter array.
Similarly,
the ordering of the eigenspaces
(using the index set $\J = \{0, 1, \dots, d\}$)
corresponds to the decreasing ordering of eigenvalues of $A_1$
in the ordering of the first parameter array.

The nonexistence results for $3$-, $4$-, $5$- and $6$-class
quotient-polynomial graphs are given
in Tables~\ref{tab:nonex3}, \ref{tab:nonex4}, \ref{tab:nonex5}
and~\ref{tab:nonex6}, respectively.
The given reasons for nonexistence include the following.
\begin{itemize}
\item handshake:
the handshake lemma is not satisfied
(i.e., $k_i p^i_{ij}$ is odd for some $i, j \in \I$),
\item multiplicities:
the multiplicities of the association scheme are nonintegral,
\item $q^h_{ij} < 0$:
the specified Krein parameter is negative (see~\cite[Theorem~2.3.2]{bcn89}),
\item absolute bound:
the absolute bound is exceeded (see~\cite[Theorem~2.3.3]{bcn89}),
\item $\tilde{k}_{\tilde{\imath}}$ or
$\tilde{p}^{\tilde{h}}_{\tilde{\imath}\tilde{\jmath}} \not\in \Z$
in $\A / \tilde{0}$:
the specified intersection number
of the quotient scheme for the specified imprimitivity set
is nonintegral,
\item conference for $\A / \tilde{0}$:
the quotient scheme for the specified imprimitivity set
has the parameters of a conference graph of infeasible order
(see~\cite[\S1.3]{bcn89}),
\item no spread in SRG~\cite{ht96}:
the $3$-class association scheme
corresponds to a strongly regular graph with a spread
(see~\cite[Proposition~2.2]{ht96}),
but a strongly regular graph with the resulting parameters
(obtained by fusing two of the classes)
does not admit a spread
(see~\cite[Theorems~2.4 and~6.1]{ht96}),
\item no solution for $(r, s, t)$:
there is no solution for triple intersection numbers for vertices $x, y, z$
such that $(x, y) \in R_r$, $(x, z) \in R_s$ and $(y, z) \in R_t$
(see~\cite[\S2.2]{gvw21}),
\item forbidden quadruple $(r, s, t; h, i, j)$:
there is a contradiction
for triple intersection numbers for vertices $w, x, y, z$
such that $(x, y) \in R_r$, $(x, z) \in R_s$, $(y, z) \in R_t$,
$(w, x) \in R_h$, $(w, y) \in R_i$ and $(w, z) \in R_j$
(see~\cite[Corollary~4.2]{gvw21}),
\item a reference:
the nonexistence condition is given as a result in~\cite{bcn89}
(in one case applied to a fusion scheme
obtained by taking the unions of the specified relations)
or the cited paper,
\item not in classification~\cite{hm19}:
no association scheme with the given parameters
appears in the classification of association schemes with few vertices
by Hanaki and Miyamoto,
\item Theorem in \S\ref{sec:nonex}:
all known feasibility conditions are satisfied,
but nonexistence is shown in Section~\ref{sec:nonex}.
\end{itemize}

\begin{table}
\centering
\footnotesize
\begin{tabular}{c|c|c}
order & parameter array & reason for nonexistence \\ \hline
$35$ & $[[12, 6, 16], [2, 3; 3]]$ & $q^1_{22} < 0$ \\
$35$ & $[[12, 6, 16], [4, 3; 3]]$ & no spread in SRG~\cite{ht96} \\
$35$ & $[[12, 16, 6], [3, 0; 8]]$ & $q^1_{22} < 0$ \\
$35$ & $[[24, 4, 6], [18, 16; 4]]$ & \cite[Prop.~1.10.5.]{bcn89} \\
$36$ & $[[12, 3, 20], [12, 3; 0]]$, $[[20, 3, 12], [20, 15; 0]]$ & absolute bound \\
$36$ & $[[15, 5, 15], [6, 2; 3]]$, $[[15, 5, 15], [6, 10; 3]]$ & $q^1_{13} < 0$ \\
$38$ & $[[18, 1, 18], [18, 9; 0]]$ & $\tilde{p}^{\tilde{1}}_{\tilde{1}\tilde{1}} \not\in \Z$ in $\A / \{0, 2\}$ \\
$39$ & $[[12, 12, 14], [5, 0; 6]]$ & multiplicities \\
$40$ & $[[14, 4, 21], [14, 6; 0]]$, $[[21, 4, 14], [21, 12; 0]]$ & absolute bound \\
$40$ & $[[18, 3, 18], [6, 8; 2]]$ & \cite[Lemma~6.1]{vd99} \\
$40$ & $[[18, 9, 12], [2, 6; 6]]$ & $q^1_{11} < 0$ \\
$40$ & $[[18, 9, 12], [10, 6; 6]]$ & no spread in SRG~\cite{ht96} \\
$44$ & $[[10, 3, 30], [10, 2; 0]]$ & $\tilde{k}_{\tilde{1}} \not\in \Z$ in $\A / \{0, 2\}$ \\
$45$ & $[[8, 8, 28], [1, 0; 2]]$ & $q^1_{13} < 0$ \\
$45$ & $[[8, 32, 4], [1, 0; 8]]$ & \cite[Prop.~4.3.3.]{bcn89} \\
$45$ & $[[12, 4, 28], [12, 3; 0]]$ & $\tilde{k}_{\tilde{1}} \not\in \Z$ in $\A / \{0, 2\}$ \\
$48$ & $[[10, 2, 35], [10, 2; 0]]$ & $\tilde{k}_{\tilde{1}} \not\in \Z$ in $\A / \{0, 2\}$ \\
$50$ & $[[14, 28, 7], [5, 0; 12]]$ & absolute bound \\
$51$ & $[[8, 2, 40], [8, 1; 0]]$, $[[40, 2, 8], [40, 35; 0]]$ & $\tilde{k}_{\tilde{1}} \not\in \Z$ in $\A / \{0, 2\}$ \\
$51$ & $[[16, 2, 32], [16, 5; 0]]$, $[[16, 32, 2], [5, 16; 0]]$, $[[32, 2, 16], [32, 22; 0]]$ & $\tilde{k}_{\tilde{1}} \not\in \Z$ in $\A / \{0, 2\}$ \\
$54$ & $[[12, 1, 40], [12, 3; 0]]$, $[[40, 1, 12], [40, 30; 0]]$ & $\tilde{p}^{\tilde{3}}_{\tilde{1}\tilde{1}} \not\in \Z$ in $\A / \{0, 2\}$ \\
$54$ & $[[26, 1, 26], [26, 13; 0]]$ & $\tilde{p}^{\tilde{1}}_{\tilde{1}\tilde{1}} \not\in \Z$ in $\A / \{0, 2\}$ \\
$56$ & $[[13, 39, 3], [3, 0; 13]]$ & handshake \\
$56$ & $[[15, 10, 30], [6, 2; 2]]$ & absolute bound \\
$56$ & $[[15, 30, 10], [4, 0; 12]]$, $[[30, 10, 15], [12, 14; 8]]$ & $q^1_{13} < 0$ \\
$56$ & $[[18, 1, 36], [18, 8; 0]]$, $[[36, 1, 18], [36, 20; 0]]$ & $q^3_{33} < 0$ \\
$56$ & $[[18, 7, 30], [18, 6; 0]]$, $[[30, 7, 18], [30, 20; 0]]$ & absolute bound \\
$56$ & $[[20, 5, 30], [4, 8; 2]]$, $[[30, 5, 20], [18, 15; 3]]$ & Fon-Der-Flaass~\cite{fdf93b} \\
$56$ & $[[24, 7, 24], [24, 12; 0]]$ & absolute bound \\
$56$ & $[[27, 27, 1], [13, 0; 27]]$ & handshake \\
$56$ & $[[39, 3, 13], [39, 30; 0]]$ & handshake \\
$58$ & $[[21, 28, 8], [15, 0; 21]]$ & \cite[Prop.~1.10.4.]{bcn89} \\
$60$ & $[[14, 3, 42], [14, 2; 0]]$, $[[42, 3, 14], [42, 36; 0]]$ & $\tilde{k}_{\tilde{1}} \not\in \Z$ in $\A / \{0, 2\}$ \\
$60$ & $[[14, 3, 42], [14, 3; 0]]$, $[[42, 3, 14], [42, 33; 0]]$ & $\tilde{k}_{\tilde{1}} \not\in \Z$ in $\A / \{0, 2\}$ \\
$60$ & $[[14, 42, 3], [2, 0; 14]]$, $[[42, 3, 14], [28, 33; 3]]$ & absolute bound \\
$60$ & $[[28, 3, 28], [28, 14; 0]]$ & $\tilde{p}^{\tilde{1}}_{\tilde{1}\tilde{1}} \not\in \Z$ in $\A / \{0, 2\}$ \\
\end{tabular}
\caption{Parameter arrays for $3$-class relational QPGs
marked as feasible in~\cite{hm23a}
which fail a known feasibility condition.}
\label{tab:nonex3}
\end{table}

\begin{remark}
\label{rem:spread}
Nonexistence for the two examples in Table~\ref{tab:nonex3}
which are ruled out by results in~\cite{ht96}
can also be shown using the technique described in Section~\ref{sec:embeddings}.
The proofs are given in the \href{https://nbviewer.org/github/jaanos/eigenspace-embeddings/blob/main/QPG3-12-35-16.ipynb}{\tt QPG3-12-35-16.ipynb}
and
\href{https://nbviewer.org/github/jaanos/eigenspace-embeddings/blob/main/QPG3-18-40-12.ipynb}{\tt QPG3-18-40-12.ipynb}
notebooks on the {\tt eigenspace-embeddings} repository~\cite{v25}.
\end{remark}

\begin{table}
\centering
\footnotesize
\begin{tabular}{c|c|c}
order & parameter array & reason for nonexistence \\ \hline
$20$ & $[[8, 2, 1, 8], [4, 0, 3; 0, 1; 1]]$ & not in classification~\cite{hm19} \\
$24$ & $[[6, 1, 6, 10], [6, 1, 0; 0, 0; 3]]$ & multiplicities \\
$27$ & $[[12, 2, 6, 6], [12, 6, 4; 0, 0; 6]]$ & absolute bound \\
$32$ & $[[18, 1, 6, 6], [18, 12, 9; 0, 0; 6]]$ & no solution for $(4, 4, 1)$ \\
$40$ & $[[8, 3, 12, 16], [8, 0, 2; 0, 0; 6]]$ & no solution for $(1, 1, 4)$ \\
$40$ & $[[9, 9, 3, 18], [4, 0, 0; 3, 1; 1]]$ & $q^1_{14} < 0$ \\
$40$ & $[[12, 3, 12, 12], [12, 0, 2; 0, 0; 8]]$ & $q^1_{14} < 0$ \\
$42$ & $[[10, 1, 10, 20], [10, 1, 0; 0, 0; 4]]$ & $q^1_{14} < 0$ \\
$42$ & $[[10, 1, 10, 20], [10, 1, 2; 0, 0; 4]]$ & no solution for $(1, 1, 3)$ \\
$42$ & $[[10, 20, 1, 10], [3, 0, 0; 0, 6; 1]]$ & \cite[Thm.~4.4.11.]{bcn89} \\
$42$ & $[[12, 5, 12, 12], [12, 0, 4; 0, 0; 8]]$ & absolute bound \\
$44$ & $[[12, 1, 10, 20], [12, 6, 3; 0, 0; 3]]$ & $q^3_{33} < 0$ \\
$45$ & $[[8, 8, 4, 24], [1, 0, 2; 2, 1; 1]]$ & Theorem~\ref{thm:p4-8-45-18} \\
$45$ & $[[12, 2, 12, 18], [12, 3, 0; 0, 0; 6]]$ & $q^1_{14} < 0$ \\
$45$ & $[[12, 4, 4, 24], [6, 0, 3; 0, 1; 2]]$ & Theorem~\ref{thm:p4-12-45-52} \\
$48$ & $[[12, 5, 12, 18], [12, 0, 4; 0, 0; 8]]$ & absolute bound \\
$48$ & $[[12, 7, 12, 16], [12, 0, 3; 0, 0; 9]]$ & absolute bound \\
$52$ & $[[12, 1, 14, 24], [12, 0, 5; 0, 0; 7]]$ & $\tilde{p}^{\tilde{4}}_{\tilde{1}\tilde{1}} \not\in \Z$ in $\A / \{0, 2\}$ \\
$54$ & $[[12, 1, 16, 24], [12, 3, 1; 0, 0; 6]]$ & no solution for $(1, 1, 4)$ \\
$54$ & $[[12, 1, 16, 24], [12, 3, 2; 0, 0; 6]]$ & $\tilde{p}^{\tilde{3}}_{\tilde{1}\tilde{1}} \not\in \Z$ in $\A / \{0, 2\}$ \\
$54$ & $[[12, 5, 12, 24], [12, 0, 2; 0, 0; 6]]$ & absolute bound \\
$54$ & $[[12, 8, 15, 18], [12, 0, 2; 0, 0; 10]]$ & absolute bound \\
$56$ & $[[9, 36, 1, 9], [2, 0, 0; 0, 8; 1]]$ & $q^4_{44} < 0$ \\
$56$ & $[[10, 18, 9, 18], [5, 0, 0; 0, 5; 5]]$ & $q^1_{13} < 0$ \\
$56$ & $[[12, 1, 12, 30], [12, 2, 0; 0, 0; 4]]$ & $q^1_{14} < 0$ \\
$57$ & $[[10, 10, 6, 30], [2, 0, 2; 5, 1; 1]]$ & $q^1_{11} < 0$ \\
$60$ & $[[10, 4, 20, 25], [10, 0, 2; 0, 0; 8]]$ & no solution for $(1, 1, 4)$ \\
$60$ & $[[12, 2, 18, 27], [12, 0, 4; 0, 0; 8]]$ & $\tilde{p}^{\tilde{4}}_{\tilde{1}\tilde{1}} \not\in \Z$ in $\A / \{0, 2\}$ \\
$60$ & $[[12, 5, 18, 24], [12, 0, 3; 0, 0; 9]]$ & no solution for $(1, 1, 4)$
\end{tabular}
\caption{Parameter arrays for $4$-class relational QPGs
marked as feasible in~\cite{hm23a}
which fail a known feasibility condition.}
\label{tab:nonex4}
\end{table}

\begin{table}
\centering
\footnotesize
\begin{tabular}{c|>{\centering\arraybackslash}p{64mm}|>{\centering\arraybackslash}p{55mm}}
order & parameter array & reason for nonexistence \\ \hline
$32$ & $[[6, 6, 1, 9, 9], [1, 0, 2, 0; 6, 0, 2; 0, 0; 2]]$ & absolute bound \\
$35$ & $[[6, 12, 2, 2, 12], [2, 0, 0, 0; 0, 6, 2; 0, 1; 0]]$ & multiplicities \\
$35$ & $[[12, 2, 2, 6, 12], [6, 0, 4, 4; 0, 0, 1; 2, 1; 2]]$ & \cite[Prop.~1.10.5.]{bcn89} for $(\{0\}, \{4\}, \{1, 5\}, \{2, 3\})$-fusion \\
$36$ & $[[8, 8, 1, 2, 16], [1, 0, 4, 2; 8, 4, 2; 0, 0; 0]]$ & not in classification~\cite{hm19} \\
$40$ & $[[6, 3, 6, 12, 12], [2, 0, 2, 0; 2, 0, 0; 0, 2; 2]]$ & $q^1_{15} < 0$ \\
$40$ & $[[6, 24, 1, 2, 6], [1, 0, 3, 0; 0, 0, 4; 0, 1; 1]]$ & forbidden quadruple $(5, 5, 2; 4, 5, 5)$ \\
$40$ & $[[12, 6, 3, 6, 12], [6, 4, 2, 0; 0, 2, 1; 0, 2; 3]]$ & $q^1_{15} < 0$ \\
$40$ & $[[12, 12, 1, 2, 12], [4, 12, 0, 3; 0, 0, 4; 0, 0; 2]]$ & $\tilde{p}^{\tilde{1}}_{\tilde{1}\tilde{1}} \not\in \Z$ in $\A / \{0, 3\}$ \\
$42$ & $[[6, 18, 1, 4, 12], [1, 6, 0, 0; 0, 0, 3; 0, 0; 2]]$, $[[12, 4, 1, 6, 18], [6, 12, 2, 4; 0, 4, 0; 0, 0; 2]]$ & absolute bound \\
$42$ & $[[8, 8, 1, 8, 16], [2, 8, 0, 2; 0, 4, 0; 0, 0; 2]]$ & $q^1_{11} < 0$ \\
$42$ & $[[10, 10, 1, 10, 10], [5, 0, 0, 4; 0, 5, 0; 0, 1; 5]]$ & conference for $\A / \{0, 3\}$ \\
$42$ & $[[10, 10, 1, 10, 10], [8, 0, 0, 1; 0, 2, 0; 0, 1; 8]]$ & multiplicities \\
$42$ & $[[12, 4, 1, 12, 12], [3, 0, 5, 2; 0, 2, 1; 0, 1; 5]]$ & absolute bound \\
$42$ & $[[12, 4, 1, 12, 12], [3, 0, 5, 3; 0, 2, 1; 0, 1; 5]]$ & absolute bound \\
$44$ & $[[9, 9, 1, 12, 12], [4, 0, 3, 0; 9, 0, 3; 0, 0; 6]]$ & multiplicities \\
$45$ & $[[6, 18, 2, 6, 12], [1, 0, 2, 0; 0, 0, 3; 0, 1; 2]]$, $[[12, 18, 2, 6, 6], [4, 6, 8, 0; 0, 0, 6; 0, 2; 4]]$ & Theorem~\ref{thm:p5-6-45-22} \\
$45$ & $[[6, 18, 2, 6, 12], [1, 6, 0, 0; 0, 0, 3; 0, 0; 3]]$, $[[12, 6, 2, 6, 18], [6, 12, 0, 4; 0, 6, 0; 0, 0; 2]]$ & no solution for $(1, 1, 2)$ \\
$45$ & $[[12, 4, 4, 12, 12], [6, 0, 6, 3; 0, 1, 1; 2, 2; 3]]$ & $q^1_{33} < 0$ \\
$48$ & $[[10, 5, 2, 10, 20], [2, 10, 0, 2; 0, 4, 0; 0, 0; 2]]$ & $q^2_{22} < 0$ \\
$48$ & $[[10, 5, 2, 10, 20], [2, 10, 0, 3; 0, 4, 0; 0, 0; 3]]$ & $q^5_{55} < 0$ \\
$48$ & $[[10, 5, 2, 10, 20], [2, 10, 1, 2; 0, 4, 0; 0, 0; 2]]$ & $q^2_{22} < 0$ \\
$48$ & $[[10, 10, 1, 6, 20], [1, 0, 5, 2; 10, 5, 2; 0, 0; 0]]$ & $q^1_{11} < 0$ \\
$48$ & $[[10, 10, 3, 4, 20], [1, 10, 0, 2; 0, 10, 2; 0, 0; 0]]$ & $q^2_{22} < 0$ \\
$48$ & $[[10, 20, 3, 4, 10], [3, 10, 0, 0; 0, 0, 6; 0, 0; 4]]$ & $q^5_{55} < 0$ \\
$48$ & $[[12, 4, 1, 6, 24], [6, 12, 0, 4; 0, 4, 0; 0, 0; 2]]$ & multiplicities \\
$48$ & $[[12, 12, 3, 8, 12], [2, 12, 0, 1; 0, 6, 6; 0, 0; 4]]$ & absolute bound \\
$48$ & $[[14, 2, 3, 7, 21], [7, 0, 4, 4; 0, 2, 0; 0, 2; 2]]$ & handshake \\
$48$ & $[[14, 4, 1, 14, 14], [7, 0, 5, 3; 0, 0, 2; 0, 1; 5]]$ & multiplicities \\
$54$ & $[[12, 4, 1, 12, 24], [3, 12, 0, 3; 0, 3, 0; 0, 0; 3]]$ & absolute bound \\
$54$ & $[[12, 24, 2, 3, 12], [3, 12, 0, 0; 0, 0, 6; 0, 0; 3]]$ & multiplicities \\
$56$ & $[[9, 9, 1, 18, 18], [1, 0, 2, 1; 9, 2, 1; 0, 0; 3]]$ & handshake \\
$56$ & $[[9, 9, 9, 10, 18], [4, 0, 0, 2; 5, 0, 0; 0, 2; 5]]$ & multiplicities \\
$56$ & $[[12, 12, 1, 6, 24], [1, 0, 6, 3; 12, 6, 3; 0, 0; 0]]$ & absolute bound \\
$56$ & $[[12, 12, 1, 6, 24], [1, 12, 0, 2; 0, 4, 4; 0, 0; 2]]$ & no solution for $(1, 1, 2)$ \\
$56$ & $[[12, 12, 1, 6, 24], [4, 12, 0, 3; 0, 4, 1; 0, 0; 2]]$ & $q^2_{22} < 0$ \\
$56$ & $[[12, 12, 3, 4, 24], [1, 4, 6, 3; 8, 6, 3; 0, 0; 0]]$ & absolute bound \\
$56$ & $[[12, 12, 3, 4, 24], [2, 4, 6, 3; 8, 6, 3; 0, 0; 0]]$, $[[12, 24, 3, 4, 12], [3, 4, 6, 0; 0, 0, 6; 0, 2; 2]]$ & $q^4_{45} < 0$ \\
$56$ & $[[12, 12, 3, 4, 24], [2, 12, 0, 3; 0, 12, 3; 0, 0; 0]]$, $[[12, 24, 3, 4, 12], [3, 12, 0, 0; 0, 0, 6; 0, 0; 4]]$ & absolute bound \\
$60$ & $[[6, 24, 1, 12, 16], [1, 6, 0, 0; 0, 0, 3; 0, 0; 3]]$ & multiplicities \\
$60$ & $[[12, 8, 3, 12, 24], [3, 12, 0, 3; 0, 6, 0; 0, 0; 3]]$ & no solution for $(1, 1, 2)$ \\
$60$ & $[[12, 12, 2, 9, 24], [1, 12, 0, 2; 0, 4, 4; 0, 0; 3]]$ & no solution for $(1, 1, 2)$ \\
$60$ & $[[12, 32, 1, 2, 12], [3, 12, 0, 0; 0, 0, 8; 0, 0; 2]]$ & $\tilde{p}^{\tilde{2}}_{\tilde{1}\tilde{1}} \not\in \Z$ in $\A / \{0, 3\}$
\end{tabular}
\caption{Parameter arrays for $5$-class relational QPGs
marked as feasible in~\cite{hm23a}
which fail a known feasibility condition.}
\label{tab:nonex5}
\end{table}

\begin{table}
\centering
\footnotesize
\begin{tabular}{c|c|c}
order & parameter array & reason for nonexistence \\ \hline
$36$ & $[[8, 8, 1, 2, 8, 8], [4, 0, 0, 3, 0; 0, 4, 0, 3; 0, 0, 1; 1, 0; 4]]$ & not in classification~\cite{hm19} \\
$40$ & $[[6, 6, 1, 6, 8, 12], [2, 6, 0, 0, 1; 0, 0, 3, 0; 0, 0, 0; 0, 3; 2]]$ & $q^2_{22} < 0$ \\
$42$ & $[[12, 12, 1, 4, 6, 6], [5, 0, 3, 6, 4; 12, 3, 4, 6; 0, 0, 0; 2, 2; 0]]$ & $q^5_{66} < 0$ \\
$48$ & $[[12, 12, 1, 2, 8, 12], [4, 12, 0, 3, 2; 0, 0, 6, 4; 0, 0, 0; 0, 2; 2]]$ & $\tilde{p}^{\tilde{5}}_{\tilde{1}\tilde{1}} \not\in \Z$ in $\A / \{0, 3\}$ \\
$48$ & $[[12, 12, 1, 2, 8, 12], [4, 0, 6, 3, 0; 0, 0, 6, 4; 0, 0, 1; 0, 1; 2]]$ & $q^1_{16} < 0$ \\
$54$ & $[[12, 6, 2, 6, 9, 18], [6, 12, 0, 0, 4; 0, 6, 0, 0; 0, 0, 0; 0, 2; 6]]$ & no solution for $(4, 4, 6)$ \\
$54$ & $[[12, 6, 3, 8, 12, 12], [6, 0, 6, 0, 4; 4, 0, 2, 0; 0, 0, 2; 4, 0; 6]]$ & $q^1_{14} < 0$ \\
$54$ & $[[12, 8, 3, 6, 12, 12], [3, 0, 8, 0, 5; 8, 0, 4, 0; 0, 0, 1; 2, 0; 6]]$ & $q^1_{11} < 0$ \\
$56$ & $[[12, 3, 4, 12, 12, 12], [4, 0, 0, 7, 3; 0, 2, 0, 0; 0, 2, 2; 3, 7; 0]]$ & absolute bound \\
$56$ & $[[12, 12, 3, 4, 12, 12], [6, 12, 0, 2, 0; 0, 0, 0, 6; 0, 0, 0; 4, 0; 6]]$ & absolute bound \\
$60$ & $[[12, 12, 2, 3, 6, 24], [3, 0, 0, 4, 0; 0, 4, 4, 3; 0, 0, 1; 0, 1; 1]]$ & $q^1_{16} < 0$ \\
$63$ & $[[12, 6, 2, 6, 18, 18], [6, 12, 0, 4, 0; 0, 6, 0, 0; 0, 0, 0; 2, 0; 6]]$ & no solution for $(4, 4, 5)$ \\
$64$ & $[[12, 12, 1, 6, 8, 24], [4, 0, 6, 0, 2; 12, 6, 0, 2; 0, 0, 0; 0, 0; 4]]$ & $q^5_{56} < 0$ \\
$64$ & $[[12, 12, 1, 8, 12, 18], [5, 0, 0, 0, 4; 0, 3, 5, 0; 0, 1, 0; 0, 4; 4]]$ & $q^2_{22} < 0$ \\
$70$ & $[[12, 12, 1, 8, 12, 24], [4, 0, 3, 1, 2; 0, 0, 4, 2; 0, 1, 0; 2, 2; 2]]$ & $q^4_{66} < 0$
\end{tabular}
\caption{Parameter arrays for $6$-class relational QPGs
marked as feasible in~\cite{hm23a}
which fail a known feasibility condition.}
\label{tab:nonex6}
\end{table}

Tables~\ref{tab:cons3}, \ref{tab:cons4}, \ref{tab:cons5} and~\ref{tab:cons6}
show the constructions of association schemes
for the parameter arrays marked as feasible in~\cite{hm23a},
together with the number of corresponding association scheme
and a reference (if applicable).
When the reference is not given,
the number of corresponding association schemes
can be deduced from the tables by Van Dam~\cite{vd99}
and the classifcation by Hanaki and Miyamoto~\cite{hm19}
(as well as other well-known results on distance-regular graphs,
see~\cite{b13} and~\cite{bcn89} for more details),
and the properties of the constructions given below.

Most of the constructions involve derivation from smaller association schemes.
Although the resulting association schemes are symmetric,
some of the building blocks are asymmetric association schemes
-- i.e.,
we replace the requirement
that the relations of the association scheme are symmetric
with the requirement that the relation set is closed under transposition.
We use the following derivations from the association schemes
$\A = (X, \R = \{R_i \mid i \in \I\})$,
$\A' = (X', \R' = \{R'_i \mid i \in \I'\})$
and $\A^{(x)} = (X^{(x)}, \R^{(x)} = \{R^{(x)}_i \mid i \in \I'\})$
($x \in X$),
where $0 \in \I$ and $R_0 = \Id_X$,
and $\A^{(x)}$ has the same parameters as $\A'$.

\begin{itemize}
\item The {\em direct product}~\cite[\S3.2]{b04}
$$
\A \times \A' = \left(X \times X', \{R_i \otimes R'_j \mid i \in \I, j \in \I'\}\right).
$$
An association scheme with the same parameters as $\A \times \A'$
is necessarily a direct product of association schemes
with the same parameters as $\A$ and $\A'$.

\item The {\em lexicographic coproduct}~\cite[\S10.6.1]{b04}
\begin{multline*}
\A[f] = \Bigg(\coprod_{x \in X} X^{(x)},
\ejcbrk
\{\{((x, y), (x', y')) \mid (x, x') \in R_i, y \in X^{(x)}, y' \in X^{(x')}\} \mid i \in \I \setminus \{0\}\} \\
{} \cup \{\{((x, y), (x, y')) \mid x \in X, (y, y') \in R^{(x)}_j\} \mid j \in \I'\}\Bigg),
\end{multline*}
where $f$ is a map from $X$
to the set of association schemes with the same parameters as $\A'$
such that $f(x) = \A^{(x)}$ ($x \in X$).
In the case when $f(x) = \A'$ for all $x \in X$,
we write $\A[f] = \A[\A']$ and call the resulting association scheme
the {\em lexicographic product}\footnote{
In the literature,
this product is known as {\em nesting}
or the {\em wreath product}~\cite[\S3.4]{b04}
and denoted by $\A \wr \A'$ or $\A / \A'$.
However, unlike the direct product of association schemes,
which is a natural generalization of the direct product of groups,
this construction is unrelated to the wreath product of groups.
Therefore,
we prefer the name {\em lexicographic product}
and adopt the notation used for the lexicographic product of graphs,
as the adjacency relation of the lexicographic product
of graphs whose adjacency relations are relations of $\A$ and $\A'$
is a union of the corresponding relations of $\A[\A']$.
}
of $\A$ and $\A'$.
An association scheme with the same parameters as $\A[f]$
is necessarily a lexicographic coproduct of association schemes
with the same parameters as $\A$ and $\A'$
(cf.~\cite{x11}, where it is assumed that $f$ is constant).

\item The $k$-th {\em Hamming power}~\cite[\S10.6.3]{b04}
$$
H(k, \A) = \left(X^k, \left\{\bigcup_{v \in {k \choose u}} \bigotimes_{j=1}^k R_{v_j} \;\middle|\; u \in \Z^\I, u \ge 0, \sum_{i=1}^k u_i = k\right\}\right),
$$
where ${k \choose u}$ is the set of all vectors from $\I^k$
in which each entry $i \in \I$ occurs $u_i$ times.
The Hamming power can be seen as a generalization of Hamming schemes
(cf.~\cite[\S1.4.3]{b04}, \cite[\S9.2]{bcn89}),
i.e.,
the Hamming scheme $H(k, n)$ is precisely the Hamming power $H(k, K_n)$
of the $1$-class association scheme on $n$ vertices.
Unlike the direct and lexicographic products,
an association scheme with the same parameters as $H(k, \A)$
is not necessarily a Hamming power
of an association scheme with the same parameters of $\A$
-- a counterexample is
the association scheme corresponding to the Shrikhande graph~\cite{s59},
which has the same parameters as $H(2, 4)$.

\item The {\em symmetrization}
$$
\A^\ddag = (X, \{R_i \cup R_i^\top \mid i \in \I\}).
$$
If $\A$ is a commutative association scheme
(i.e., $p^h_{ij} = p^h_{ji}$ for all $h, i, j \in \I$),
then $\A^\ddag$ is a symmetric association scheme.
\end{itemize}

The following association schemes are used as building blocks.
Unless noted otherwise, these association schemes are symmetric.
\begin{itemize}
\item $K_n$: the $1$-class association scheme on $n$ vertices.
\item $Z_n$: the cyclic group on $n$ vertices
as the corresponding thin association scheme.
The association scheme $Z_n$ is commutative,
but is only symmetric when $n \le 2$.
\item $C_n$: the cyclic scheme on $n$ vertices (i.e., $C_n = Z_n^\ddag$).
\item $\Had(4n)$: the association schemes
corresponding to the incidence graphs of square $2$-$(4n-1, 2n, n)$ designs
associated to Hadamard matrices of order $4n$.
\item $\GH(s, t)$: the association schemes corresponding to the point graphs
of generalized hexagons of order $(s, t)$ (see~\cite[\S6.5]{bcn89}).
\item $J(n, k)$: the Johnson scheme of $k$-subsets of a set of size $n$
(see~\cite[\S1.4.2]{b04}, \cite[\S9.1]{bcn89}).
\item $\Pair(n)$: the association scheme
of ordered $2$-subsets of a set of size $n$,
with classes corresponding to pairs matching in one coordinate,
pairs matching in different coordinates,
disjoint pairs and reversed pairs (see~\cite[\S5.5]{b04}).
\item $\Cyc(q, r)$: the cyclotomic scheme
$$
\left(\F_q, \left\{\Id_{\F_q}, \left\{(x, x + \gamma^{i + rj}) \;\middle|\; x \in \F_q, j = 1, \dots, {q-1 \over r}\right\} \;\middle|\; i = 1, \dots, r\right\}\right),
$$
where $q$ is a prime power, $r$ divides $q-1$,
and $\gamma$ generates the multiplicative group $\F_q^*$
(see~\cite[\S2.4]{d73a}).
The association scheme $\Cyc(q, r)$ is commutative,
and is symmetric precisely when $q$ is even or ${q-1 \over r}$ is even.
\item Hyperbolic quadric in $\PG(3, q)$:
the association scheme of the points of a hyperbolic quadric
in the projective geometry $\PG(3, q)$
(see~\cite[\S12.2]{bcn89} for the construction).
\item Locally $\Cyc(q, r)$:
the association scheme corresponding to
the distance-regular antipodal $r$-cover of $K_{q+1}$
whose local graphs are the graphs of the association scheme $\Cyc(q, r)$
(see~\cite[Proposition~12.5.3]{bcn89} for the construction).
\item Locally $\Conf(q)$:
the association schemes corresponding to
distance-regular antipodal double covers of $K_{q+1}$
whose local graphs are (not necessarily mutually isomorphic)
$\SRG(q, {q-1 \over 2}, {q-5 \over 4}, {q-1 \over 4})$ conference graphs.
They are in one-to-one correspondence with $2$-graphs on $q+1$ vertices.
When $q$ is a prime power with $q \equiv 1 \pmod{4}$,
an example is given by the locally $\Cyc(q, 2)$ association scheme.
\item Named graphs: the association scheme corresponding to the
(distance-regular) named graph (see~\cite{bcn89}).
\item $\SRG(n, k, \lambda, \mu)$:
the association schemes corresponding to
strongly regular graphs with parameters $(n, k, \lambda, \mu)$
(when there are multiple such graphs not sharing a common construction).
\item $\Cay(G, S)$: the association scheme corresponding to the
Cayley graph of the group $G$ with the connecting set $S$,
under the assumption that it is quotient-polynomial.
The association scheme $\Cay(G, S)$ is symmetric
if $S$ is closed under inversion in $G$.
\end{itemize}

\begin{table}
\centering
\footnotesize
\begin{tabular}{c|>{\centering\arraybackslash}p{66mm}|>{\centering\arraybackslash}p{40mm}|c|c}
order & parameter array & construction & \# & reference \\ \hline
$35$ & $[[12, 4, 18], [3, 4; 2]]$ & $\SRG(35, 16, 6, 8) \setminus \text{spread}$ & $\ge 1$ & \cite{ht96} \\
$36$ & $[[17, 17, 1], [8, 0; 17]]$ & locally $\Cyc(17, 2)$ & $1$ \\
$36$ & $[[20, 5, 10], [8, 12; 4]]$ & Sylvester & $1$ \\
$37$ & $[[12, 12, 12], [4, 5; 3]]$ & $\Cyc(37, 3)$ & $\ge 1$ \\
$39$ & $[[18, 2, 18], [18, 9; 0]]$ & $\Cyc(13, 2)[K_3]$ & $1$ \\
$39$ & $[[24, 2, 12], [12, 22; 2]]$ & $K_3 \times K_{13}$ & $1$ \\
$42$ & $[[13, 26, 2], [4, 0; 13]]$ & locally $\Cyc(13, 3)$ & $\ge 1$ \\
$43$ & $[[14, 14, 14], [4, 6; 4]]$ & $\Cyc(43, 3)$ & $\ge 1$ \\
$46$ & $[[12, 22, 11], [6, 0; 12]]$ & $\Had(24)$ & $582$ \\
$51$ & $[[16, 32, 2], [5, 0; 16]]$, $[[32, 2, 16], [16, 20; 2]]$ & locally $\Cyc(16, 3)$ & $\ge 1$ \\
$51$ & $[[24, 2, 24], [24, 12; 0]]$ & $\Cyc(17, 2)[K_3]$ & $1$ \\
$51$ & $[[32, 2, 16], [16, 30; 2]]$ & $K_3 \times K_{17}$ & $1$ \\
$52$ & $[[18, 6, 27], [6, 8; 2]]$, $[[27, 6, 18], [18, 12; 3]]$ & $\GH(3, 1)$ & $1$ \\
$52$ & $[[20, 1, 30], [20, 8; 0]]$, $[[30, 1, 20], [30, 18; 0]]$ & $\SRG(26, 10, 3, 4)[K_2]$ & $10$ \\
$52$ & $[[24, 3, 24], [24, 12; 0]]$ & $\Cyc(13, 2)[K_4]$ & $1$ \\
$52$ & $[[25, 25, 1], [12, 0; 25]]$ & locally $\Conf(25)$ & $4$ \\
$52$ & $[[25, 25, 1], [24, 0; 25]]$ & $K_2 \times K_{26}$ & $1$ \\
$52$ & $[[36, 3, 12], [24, 33; 3]]$ & $K_4 \times K_{13}$ & $1$ \\
$54$ & $[[14, 26, 13], [7, 0; 14]]$ & $\Had(28)$ & $105041$ \\
$54$ & $[[20, 1, 32], [20, 10; 0]]$, $[[32, 1, 20], [32, 16; 0]]$ & $\text{Schläfli}[K_2]$ & $1$ \\
$54$ & $[[24, 5, 24], [24, 12; 0]]$ & $H(2, 3)[K_6]$ & $1$ \\
$54$ & $[[26, 26, 1], [25, 0; 26]]$ & $K_2 \times K_{27}$ & $1$ \\
$54$ & $[[34, 2, 17], [17, 32; 2]]$ & $K_3 \times K_{18}$ & $1$ \\
$54$ & $[[40, 5, 8], [32, 35; 5]]$ & $K_6 \times K_9$ & $1$ \\
$55$ & $[[22, 10, 22], [22, 11; 0]]$ & $C_5[K_{11}]$ & $1$ \\
$55$ & $[[40, 4, 10], [30, 36; 4]]$ & $K_5 \times K_{11}$ & $1$ \\
$56$ & $[[10, 15, 30], [4, 1; 3]]$, $[[30, 10, 15], [18, 16; 6]]$ & $J(8, 3)$ & $1$ \\
$56$ & $[[24, 1, 30], [24, 8; 0]]$ & $J(8, 2)[K_2]$, $\text{Chang}_i[K_2]$ ($i = 1, 2, 3$) & $4$ \\
$56$ & $[[27, 27, 1], [10, 0; 27]]$, $[[27, 27, 1], [16, 0; 27]]$ & Gosset & $1$ \\
$56$ & $[[27, 27, 1], [26, 0; 27]]$ & $K_2 \times K_{28}$ & $1$ \\
$56$ & $[[39, 3, 13], [26, 36; 3]]$ & $K_4 \times K_{14}$ & $1$ \\
$56$ & $[[42, 6, 7], [35, 36; 6]]$ & $K_7 \times K_8$ & $1$ \\
$57$ & $[[20, 6, 30], [10, 6; 2]]$, $[[30, 6, 20], [15, 18; 3]]$ & Perkel & $1$ & \cite{cd05} \\
$57$ & $[[36, 2, 18], [18, 34; 2]]$ & $K_3 \times K_{19}$ & $1$ \\
$58$ & $[[28, 1, 28], [28, 14; 0]]$ & locally $\Conf(29)$ & $41$ \\
$58$ & $[[28, 28, 1], [27, 0; 28]]$ & $K_2 \times K_{29}$ & $1$ \\
$60$ & $[[11, 44, 4], [2, 0; 11]]$ & locally $\Cyc(11, 5)$ & $\ge 1$ \\
$60$ & $[[15, 20, 24], [3, 5; 5]]$, $[[20, 15, 24], [8, 5; 5]]$, $[[24, 15, 20], [8, 12; 6]]$ & hyperbolic quadric in $\PG(3, 5)$ & $\ge 1$ \\
$60$ & $[[18, 5, 36], [18, 6; 0]]$, $[[36, 5, 18], [36, 24; 0]]$ & $\text{Petersen}[K_6]$ & $1$ \\
$60$ & $[[24, 3, 32], [24, 12; 0]]$, $[[32, 3, 24], [32, 16; 0]]$ & $J(6, 2)[K_4]$ & $1$ \\
$60$ & $[[24, 11, 24], [24, 12; 0]]$ & $C_5[K_{12}]$ & $1$ \\
$60$ & $[[38, 2, 19], [19, 24; 2]]$ & locally $\Cyc(19, 3)$ & $\ge 1$ \\
$60$ & $[[42, 3, 14], [28, 39; 3]]$ & $K_4 \times K_{15}$ & $1$ \\
$60$ & $[[44, 4, 11], [33, 40; 4]]$ & $K_5 \times K_{12}$ & $1$ \\
$60$ & $[[45, 5, 9], [36, 40; 5]]$ & $K_6 \times K_{10}$ & $1$
\end{tabular}
\caption{Parameter arrays for $3$-class relational QPGs
marked as feasible in~\cite{hm23a}
for which constructions are known.}
\label{tab:cons3}
\end{table}

\begin{table}
\centering
\footnotesize
\begin{tabular}{c|>{\centering\arraybackslash}p{46mm}|c|c|c}
order & parameter array & construction & \# & reference \\ \hline
$40$ & $[[8, 4, 3, 24], [2, 0, 2; 0, 1; 1]]$ & {\em see Remark~\ref{rem:p4-8-40-15}} & $1$ & \cite{bbb08} \\
$41$ & $[[10, 10, 10, 10], [4, 3, 2; 2, 2; 2]]$ & $\Cyc(41, 4)$ & $\ge 1$ \\
$42$ & $[[8, 1, 16, 16], [8, 2, 0; 0, 0; 4]]$ & $\GH(2, 1)[K_2]$ & $1$ \\
$42$ & $[[9, 2, 12, 18], [9, 0, 3; 0, 0; 6]]$, $[[12, 2, 9, 18], [12, 0, 6; 0, 0; 6]]$ & $\text{Heawood}[K_3]$ & $1$ \\
$42$ & $[[12, 5, 12, 12], [12, 0, 6; 0, 0; 6]]$ & $C_7[K_6]$ & $1$ \\
$44$ & $[[10, 1, 12, 20], [10, 0, 4; 0, 0; 6]]$, $[[12, 1, 10, 20], [12, 0, 6; 0, 0; 6]]$ & $(\text{square $2$-$(11, 5, 2)$ design})[K_2]$ & $1$ \\
$52$ & $[[8, 1, 18, 24], [8, 0, 2; 0, 0; 6]]$ & $\GH(1, 3)[K_2]$ & $1$ \\
$54$ & $[[9, 24, 2, 18], [3, 0, 0; 0, 8; 1]]$ & symmetric $(3, 3)$-nets & $4$ & \cite{mt00} \\
$56$ & $[[12, 3, 16, 24], [12, 0, 4; 0, 0; 8]]$ & $\text{Heawood}[K_4]$ & $1$
\end{tabular}
\caption{Parameter arrays for $4$-class relational QPGs
marked as feasible in~\cite{hm23a}
for which constructions are known.}
\label{tab:cons4}
\end{table}

\begin{table}
\centering
\footnotesize
\begin{tabular}{c|>{\centering\arraybackslash}p{64mm}|>{\centering\arraybackslash}p{47mm}|c|c}
order & parameter array & construction & \# & reference \\ \hline
$27$ & $[[6, 6, 2, 6, 6], [3, 6, 0, 0; 0, 0, 3; 0, 0; 3]]$ & $C_9[K_3]$ & $1$ \\
$36$ & $[[8, 16, 1, 2, 8], [2, 0, 4, 0; 0, 0, 4; 0, 1; 1]]$ & $H(2, K_3[K_2])$ & $1$ \\
$40$ & $[[12, 2, 1, 12, 12], [6, 0, 4, 1; 0, 0, 1; 0, 1; 4]]$ & {\em see Remark~\ref{rem:p5-12-40-2}} & $1$ & Thm.~\ref{thm:p5-12-40-2} \\
$40$ & $[[12, 12, 1, 2, 12], [4, 12, 0, 4; 0, 0, 4; 0, 0; 2]]$ & $\Pair(5)[K_2]$ & $1$ \\
$40$ & $[[14, 2, 2, 7, 14], [7, 0, 12, 6; 0, 0, 1; 2, 1; 0]]$ & $K_8 \times C_5$ & $1$ \\
$41$ & $[[8, 8, 8, 8, 8], [3, 2, 2, 0; 1, 2, 2; 1, 2; 1]]$ & $\Cyc(41, 5)$ & $\ge 1$ \\
$42$ & $[[10, 10, 1, 10, 10], [5, 0, 4, 0; 10, 0, 4; 0, 0; 6]]$, $[[10, 10, 1, 10, 10], [6, 0, 0, 3; 0, 4, 0; 0, 1; 6]]$ & $K_2 \times J(7, 2)$ & $1$ \\
$42$ & $[[12, 12, 1, 4, 12], [4, 12, 0, 2; 0, 6, 6; 0, 0; 2]]$ & $(Z_3 \times \Cyc(7, 2))^\ddag[K_2]$ & $1$ \\
$45$ & $[[6, 4, 4, 12, 18], [3, 0, 0, 1; 0, 1, 0; 2, 0; 2]]$, $[[12, 4, 4, 6, 18], [9, 6, 0, 4; 0, 2, 0; 4, 0; 2]]$ & {\em see Remark~\ref{rem:p5-6-45-5}} & $1$ & Thm.~\ref{thm:p5-6-45-5} \\
$45$ & $[[10, 10, 4, 10, 10], [5, 10, 0, 0; 0, 0, 5; 0, 0; 5]]$ & $C_9[K_5]$ & $1$ \\
$48$ & $[[10, 2, 5, 10, 20], [5, 0, 4, 2; 2, 0, 0; 0, 2; 3]]$ & $K_3 \times \text{Clebsch}$ & $1$ \\
$48$ & $[[12, 2, 6, 9, 18], [6, 4, 4, 2; 2, 0, 0; 0, 2; 4]]$ & $K_3 \times H(2, 4)$, $K_3 \times \text{Shrikhande}$ & $2$ \\
$50$ & $[[8, 8, 1, 16, 16], [3, 0, 2, 0; 8, 0, 2; 0, 0; 6]]$ & $K_2 \times H(2, 5)$ & $1$ \\
$50$ & $[[12, 4, 3, 6, 24], [9, 0, 4, 3; 4, 0, 0; 0, 1; 2]]$ & $K_5 \times \text{Petersen}$ & $1$ \\
$50$ & $[[12, 12, 1, 12, 12], [4, 12, 2, 0; 0, 4, 4; 0, 0; 4]]$ & $\Cyc(25, 4)[K_2]$ & $1$ \\
$51$ & $[[12, 12, 2, 12, 12], [6, 12, 3, 0; 0, 3, 3; 0, 0; 3]]$ & $\Cyc(17, 4)[K_3]$ & $1$ \\
$54$ & $[[10, 10, 1, 16, 16], [1, 0, 5, 0; 10, 0, 5; 0, 0; 5]]$ & $K_2 \times \text{Schläfli}$ & $1$ \\
$54$ & $[[12, 12, 5, 12, 12], [6, 12, 0, 0; 0, 0, 6; 0, 0; 6]]$ & $C_9[K_6]$ & $1$ \\
$56$ & $[[6, 12, 1, 12, 24], [2, 6, 0, 0; 0, 0, 2; 0, 0; 2]]$, $[[12, 12, 1, 6, 24], [4, 12, 0, 2; 0, 0, 4; 0, 0; 2]]$ & $\text{Coxeter}[K_2]$ & $1$ \\
$56$ & $[[12, 12, 1, 15, 15], [6, 0, 4, 0; 12, 0, 4; 0, 0; 8]]$ & $K_2 \times J(8, 2)$, $K_2 \times \text{Chang}_i$ ($i = 1, 2, 3$) & $4$
\end{tabular}
\caption{Parameter arrays for $5$-class relational QPGs
marked as feasible in~\cite{hm23a}
for which constructions are known.}
\label{tab:cons5}
\end{table}

\begin{table}
\centering
\footnotesize
\begin{tabular}{c|>{\centering\arraybackslash}p{53mm}|c|c}
order & parameter array & construction & \# \\ \hline
$45$ & $[[12, 6, 2, 6, 6, 12],$ $[6, 12, 6, 0, 3; 0, 0, 6, 0; 0, 0, 0; 0, 3; 3]]$ & $(K_3 \times C_5)[K_3]$ & $1$ \\
$52$ & $[[12, 12, 1, 2, 12, 12],$ $[6, 12, 0, 4, 0; 0, 0, 0, 6; 0, 0, 0; 2, 0; 6]]$ & $(K_2 \times \Cyc(13, 2))[K_2]$ & $1$ \\
$56$ & $[[12, 12, 1, 6, 12, 12],$ $[4, 12, 0, 2, 0; 0, 4, 4, 2; 0, 0, 0; 2, 2; 4]]$ & $\Cay(Z_{14} \times Z_2, \{(\pm 1, 0), (\pm 2, 1), (\pm 3, 1)\})[K_2]$ & $1$
\end{tabular}
\caption{Parameter arrays for $6$-class relational QPGs
marked as feasible in~\cite{hm23a}
for which constructions are known.}
\label{tab:cons6}
\end{table}

\begin{remark} \label{rem:p4-8-40-15}
The parameter array $[[8, 4, 3, 24], [2, 0, 2; 0, 1; 1]]$
uniquely determines a quo\-ti\-ent-polynomial graph $\G_1 = (X, R_1)$
which is derived from a spherical code found by Smith,
with the following construction given by Conway and Sloane
(cf.~\cite{bbcgks09,s00}).
Uniqueness is due to Bannai, Bannai and Bannai~\cite{bbb08}.
An alternative proof of uniqueness
paralleling the proofs of the results
in Sections~\ref{sec:nonex} and~\ref{sec:uniq}
is given in the
\href{https://nbviewer.org/github/jaanos/eigenspace-embeddings/blob/main/QPG4-8-40-15.ipynb}{\tt QPG4-8-40-15.ipynb}
notebook on the {\tt eigenspace-embeddings} repository~\cite{v25}.

Let $\D$ be the set of all symmetric relations $R \subseteq (\F_5^*)^2$
such that $\Delta = (\F_5^*, R)$ is a graph
with the degrees of the vertices $1$ and $2$
having the same parity as the number of edges of $\Delta$
and the degrees of the vertices $3$ and $4$
having different parity from the number of edges of $\Delta$.
Note that there are precisely $8$ such graphs, so $|\D| = 8$.
We may then define the unit vectors
$u^{(h, R)} \in \RR^{\F_5 \choose 2}$ ($h \in \F_5$, $R \in \D$)
such that
$$
u^{(h, R)}_{\{i+h, j+h\}} = \begin{cases}
0 & \text{if $0 \in \{i, j\}$,} \\
-{1 \over \sqrt{6}} & \text{if $(i, j) \in R$, and} \\
{1 \over \sqrt{6}} & \text{otherwise}
\end{cases}
\qquad {\textstyle (\{i, j\} \in {\F_5 \choose 2})}.
$$
These vectors can be viewed as a spherical representation
of the corresponding association scheme
$\A = (X, \R = \{R_i \mid 0 \le i \le 4\})$,
with two vertices being in relations $R_0, R_1, R_2, R_3, R_4$
when their corresponding vectors have inner products equal to
$1, -{1 \over 2}, 0, -{1 \over 3}, {1 \over 6}$, respectively.

The graph $\G_1$ is an $8$-regular arc-transitive graph
of diameter $3$ and girth $4$.
It has appeared in a census of edge-girth-regular graphs~\cite{gj24,gj25},
as each edge lies on precisely seven $4$-cycles,
and in a census of rotary maps~\cite{c12,cp24}
as a graph polyhedrally embedding as a chiral rotary map
on the orientable surface of genus $21$.

The corresponding association scheme $\A$
can be reconstructed from $\G_1$ as follows.
The relation $R_1$ is the adjacency relation of $\G_1$.
For each vertex $x$ of $\G_1$,
there are precisely three vertices at distance $3$ from $x$,
and they are also mutually at distance $3$
-- together with $x$ they form a $R_3$-clique $Y$.
There are precisely four vertices at distance $2$ from all vertices of $Y$
-- these vertices are in relation $R_2$ with $x$
(and with other vertices of $Y$).
Finally,
the remaining vertices are adjacent to a vertex of $Y$ distinct from $x$
-- these are in relation $R_4$ with $x$.

The group of automorphisms of $\G_1$ has order $1920$
and is isomorphic to a semidirect product $Z_2^4 \rtimes S_5$
-- this also holds for the group of automorphisms of $\A$.
Its natural action on $X^2$
preserves the partition $\R$
-- i.e., each relation of $\A$ corresponds to an orbit of the action.
\end{remark}

%% file: embeddings-arxiv.bbl
\begin{thebibliography}{10}

\bibitem{b04}
R.~A. Bailey.
\newblock {\em Association Schemes}.
\newblock Cambridge University Press, 2004.
\newblock \doi{10.1017/CBO9780511610882}.

\bibitem{bbcgks09}
B.~Ballinger, G.~Blekherman, H.~Cohn, N.~Giansiracusa, E.~Kelly, and
  A.~Schürmann.
\newblock Experimental study of energy-minimizing point configurations on
  spheres.
\newblock {\em Experiment. Math.}, 18:257--283, 2009.
\newblock \doi{10.1080/10586458.2009.10129052}.

\bibitem{bbb08}
E.~Bannai, E.~Bannai, and H.~Bannai.
\newblock Uniqueness of certain association schemes.
\newblock {\em European J. Combin.}, 29:1379--1395, 2008.
\newblock \doi{10.1016/j.ejc.2007.06.016}.

\bibitem{b11}
A.~E. Brouwer.
\newblock Parameters of distance-regular graphs, 2011.
\newblock \url{http://www.win.tue.nl/~aeb/drg/drgtables.html}.

\bibitem{b13}
A.~E. Brouwer.
\newblock Strongly regular graphs, 2013.
\newblock \url{http://www.win.tue.nl/~aeb/graphs/srg/srgtab.html}.

\bibitem{bcn89}
A.~E. Brouwer, A.~M. Cohen, and A.~Neumaier.
\newblock {\em Distance-regular graphs}, volume~18 of {\em Ergebnisse der
  Mathematik und ihrer Grenzgebiete (3) [Results in Mathematics and Related
  Areas (3)]}.
\newblock Springer-Verlag, Berlin, 1989.
\newblock \doi{10.1007/978-3-642-74341-2}.

\bibitem{c12}
M.~Conder.
\newblock Chiral regular maps (on orientable surfaces) with up to $1000$ edges,
  2012.
\newblock
  \url{https://www.math.auckland.ac.nz/~conder/ChiralMapsWithUpTo1000Edges.txt}.

\bibitem{cp24}
M.~Conder and P.~Potočnik.
\newblock Census of rotary maps by genus and number of edges, 2024.
\newblock \url{https://rotarymaps.graphsym.net/}.

\bibitem{cd05}
K.~Coolsaet and J.~Degraer.
\newblock A computer-assisted proof of the uniqueness of the {P}erkel graph.
\newblock {\em Des. Codes Cryptogr.}, 34(2--3):155--171, 2005.
\newblock \doi{10.1007/s10623-004-4852-9}.

\bibitem{vd99}
E.~R. {\noop{Dam}}{van Dam}.
\newblock Three-class association schemes.
\newblock {\em J. Algebraic Combin.}, 10(1):69--107, 1999.
\newblock \doi{10.1023/A:1018628204156}.

\bibitem{d73a}
P.~Delsarte.
\newblock An algebraic approach to the association schemes of coding theory.
\newblock {\em Philips Res. Rep. Suppl.}, (10):vi+97, 1973.

\bibitem{f16}
M.~{\` A}. Fiol.
\newblock Quotient-polynomial graphs.
\newblock {\em Linear Algebra Appl.}, 488:363--376, 2016.
\newblock \doi{10.1016/j.laa.2015.09.053}.

\bibitem{fdf93b}
D.~G. Fon-Der-Flaass.
\newblock There exists no distance-regular graph with intersection array $(5,
  4, 3; 1, 1, 2)$.
\newblock {\em European J. Combin.}, 14(5):409--412, 1993.
\newblock \doi{10.1006/eujc.1993.1045}.

\bibitem{g24}
The GAP~Group.
\newblock {\em GAP -- Groups, Algorithms, and Programming, Version 4.13.1},
  2024.
\newblock \url{http://www.gap-system.org/}.

\bibitem{glms25}
A.~L. Gavrilyuk, J.~Lansdown, A.~Munemasa, and S.~Suda.
\newblock Roux schemes which carry association schemes locally, 2025.
\newblock \arxiv{2507.18960}.

\bibitem{gs23}
A.~L. Gavrilyuk and S.~Suda.
\newblock Uniqueness of an association scheme related to the {W}itt design on
  $11$ points.
\newblock {\em Des. Codes Cryptogr.}, 92:205--209, 2023.
\newblock \doi{10.1007/s10623-023-01303-8}.

\bibitem{gvw21}
A.~L. Gavrilyuk, J.~Vidali, and J.~S. Williford.
\newblock On few-class {$Q$}-polynomial association schemes: feasible
  parameters and nonexistence results.
\newblock {\em Ars Math. Contemp.}, 20:103--127, 2021.
\newblock \doi{10.26493/1855-3974.2101.b76}.

\bibitem{gj24}
J.~Goedgebeur and J.~Jooken.
\newblock Exhaustive generation of edge-girth-regular graphs, 2024.
\newblock \url{https://github.com/JorikJooken/edgeGirthRegularGraphs}.

\bibitem{gj25}
J.~Goedgebeur and J.~Jooken.
\newblock Exhaustive generation of edge-girth-regular graphs.
\newblock {\em Exp. Math.}, 2025.
\newblock \doi{10.1080/10586458.2025.2491675}.

\bibitem{ht96}
W.~H. Haemers and V.~D. Tonchev.
\newblock Spreads in strongly regular graphs.
\newblock {\em Des. Codes Cryptogr.}, 8:145--157, 1996.
\newblock \doi{10.1007/BF00130574}.

\bibitem{hm19}
A.~Hanaki and I.~Miyamoto.
\newblock Classification of association schemes with small vertices, 2019.
\newblock \url{https://math.shinshu-u.ac.jp/~hanaki/as/}.

\bibitem{hm23a}
A.~Herman and R.~Maleki.
\newblock {QPG} database, 2023.
\newblock \url{https://github.com/RoghayehMaleki/QPGdatabase-}.

\bibitem{hm23}
A.~Herman and R.~Maleki.
\newblock The search for small association schemes with noncyclotomic
  eigenvalues.
\newblock {\em Ars Math. Contemp.}, 23(3):P3.02, 2023.
\newblock \doi{10.26493/1855-3974.2724.83d}.

\bibitem{hm24}
A.~Herman and R.~Maleki.
\newblock Parameters of quotient-polynomial graphs.
\newblock {\em Graphs Combin.}, 40:60, 2024.
\newblock \doi{10.1007/s00373-024-02789-2}.

\bibitem{jk07a}
T.~Junttila and P.~Kaski.
\newblock Engineering an efficient canonical labeling tool for large and sparse
  graphs.
\newblock In D.~Applegate, G.~S. Brodal, D.~Panario, and R.~Sedgewick, editors,
  {\em Proceedings of the Ninth Workshop on Algorithm Engineering and
  Experiments and the Fourth Workshop on Analytic Algorithms and
  Combinatorics}, pages 135--149, 2007.
\newblock \doi{10.1137/1.9781611972870.13}.

\bibitem{bliss21}
T.~Junttila and P.~Kaski.
\newblock The {\tt bliss} tool, 2021.
\newblock \url{https://users.aalto.fi/~tjunttil/bliss/}.

\bibitem{glpk20}
A.~Makhorin.
\newblock {GLPK} ({GNU} {L}inear {P}rogramming {K}it) v5.0.p1, 2020.
\newblock \url{http://www.gnu.org/software/glpk/}.

\bibitem{mt00}
V.~C. Mavron and V.~D. Tonchev.
\newblock On symmetric nets and generalized {H}adamard matrices from affine
  designs.
\newblock {\em J. Geom.}, 67:180--187, 2000.
\newblock \doi{10.1007/BF01220309}.

\bibitem{m90}
B.~D. McKay.
\newblock {n}auty {U}ser's guide ({V}ersion 1.5).
\newblock Technical Report TR-CS-90-02, Australian National University,
  Department of Computer Science, 1990.
\newblock \url{http://cs.anu.edu.au/~bdm/nauty/}.

\bibitem{pari23}
The PARI~Group, Univ. Bordeaux.
\newblock {\em PARI/GP version \texttt{2.15.4}}, 2023.
\newblock Available from \url{http://pari.math.u-bordeaux.fr/}.

\bibitem{px89}
C.~E. Praeger and M.-Y. Xu.
\newblock A characterization of a class of symmetric graphs of twice prime
  valency.
\newblock {\em European J. Combin.}, 10:91--102, 1989.
\newblock \doi{10.1016/S0195-6698(89)80037-X}.

\bibitem{sage24}
{\noop{Sage}}{The Sage Developers}.
\newblock {\em {S}age{M}ath, the {S}age {M}athematics {S}oftware {S}ystem
  ({V}ersion 10.6)}, 2024.
\newblock \url{http://www.sagemath.org}, \doi{10.5281/zenodo.8042260}.

\bibitem{s59}
S.~S. Shrikhande.
\newblock The uniqueness of the {$L_2$} association scheme.
\newblock {\em Ann. Math. Statist.}, 30(3):781--798, 1959.
\newblock \doi{10.1214/aoms/1177706207}.

\bibitem{s00}
N.~J.~A. Sloane.
\newblock Spherical codes, 2000.
\newblock \url{http://neilsloane.com/packings/}.

\bibitem{v18a}
J.~Vidali.
\newblock Using symbolic computation to prove nonexistence of distance-regular
  graphs.
\newblock {\em Electron. J. Combin.}, 25(4):P4.21, 2018.
\newblock \doi{10.37236/7763}.

\bibitem{v19}
J.~Vidali.
\newblock {\tt jaanos/sage-drg}: {\tt sage-drg} {S}age package v0.9, 2019.
\newblock \url{https://github.com/jaanos/sage-drg/},
  \doi{0.5281/zenodo.3350856}.

\bibitem{v25}
J.~Vidali.
\newblock {\tt jaanos/eigenspace-embeddings}, 2025.
\newblock \url{https://github.com/jaanos/eigenspace-embeddings/}.

\bibitem{w17}
J.~S. Williford.
\newblock Homepage, 2017.
\newblock \url{https://www.uwyo.edu/jwilliford/}, mirror available at
  \url{https://jaanos.github.io/tables/}.

\bibitem{x11}
B.~Xu.
\newblock Characterizations of wreath products of one-class association
  schemes.
\newblock {\em J. Combin. Theory Ser. A}, 118:1907--1914, 2011.
\newblock \doi{10.1016/j.jcta.2011.03.005}.

\end{thebibliography}
